\documentclass{elsarticle}
 \usepackage{graphicx,color}
 \usepackage{epstopdf}
 \usepackage{caption}
\usepackage{subcaption}
\usepackage{color}
\usepackage{natbib}

\usepackage[latin2]{inputenc}
\usepackage{amsmath,amssymb,amsbsy,amsthm,amsfonts,a4wide,tikz,pgf,ulem}

\usepackage{soul,xcolor}

\DeclareMathOperator*{\esssup}{ess\,sup}

\def\tens#1{\pmb{\mathsf{#1}}}
\def\vec#1{\boldsymbol{#1}}

\def\abs#1{\left| #1 \right|}

\def\diver{\mathop{\mathrm{div}}\nolimits}

\def\R{\mathbb{R}}
\def\Rsym{\mathbb{R}^{d\times d}_{\textrm{sym}}}

\def\bx{\vec{x}}

\def\bef{\vec{f}}

\def\bW{\tens{W}}

\def\bT{\tens{T}}
\def\bC{\tens{C}}
\def\bD{\tens{D}}
\def\bG{\tens{G}}
\def\bF{\tens{F}}
\def\bB{\tens{B}}

\def\bV{\tens{V}}
\def\bU{\tens{U}}
\def\bL{\tens{L}}
\def\bR{\tens{R}}

\def\bI{\tens{I}}
\def\bH{\tens{H}}
\def\beps{\tens{\varepsilon}}

\def\bv{\vec{v}}

\def\b0{\vec{0}}
\def\bom{\vec{\omega}}
\def\bw{\vec{w}}

\def\bu{\vec{u}}

\newcommand{\dx}{\,\mathrm{d}x}
\newcommand{\dt}{\,\mathrm{d}t}
\newcommand{\dtau}{\,\mathrm{d}\tau}

\newcommand{\ddt}{\frac{\mathrm{d}}{\mathrm{d}t}}

\newtheorem{theorem}{Theorem}[section]

\newtheorem{lemma}[theorem]{Lemma}

\newtheorem{remark}[theorem]{Remark}

\journal{}

\begin{document}

\begin{frontmatter}

 \title{Existence and uniqueness of global weak solutions to strain-limiting viscoelasticity with Dirichlet boundary data\tnoteref{t1,t2}}
\tnotetext[t1]{M.~Bul\'{\i}\v{c}ek's work is supported by the project 20-11027X financed by GA\v{C}R.  M.~Bul\'{\i}\v{c}ek is a  member of the Ne\v{c}as Center for Mathematical Modeling.
V.~Patel is supported by the UK Engineering and Physical Sciences Research Council [EP/L015811/1].}

\author[Bulicek-address]{Miroslav Bul\'{i}\v{c}ek}
\ead{mbul8060@karlin.mff.cuni.cz}
\author[Ox-address]{Victoria Patel}
\ead{victoria.patel@maths.ox.ac.uk}
\author[Sengul-address]{Yasemin \c Seng\"{u}l\corref{cor1}}
\cortext[cor1]{Corresponding author. Tel: +90 216 483 9541}
\ead{yasemin.sengul@sabanciuniv.edu}
\author[Ox-address]{Endre S\"{u}li}
\ead{endre.suli@maths.ox.ac.uk}

\address[Bulicek-address]{Charles University, Faculty of Mathematics and Physics,  Sokolovsk\'{a}~83, 186~75 Praha~8, Czech Republic.}
\address[Ox-address]{Mathematical Institute, University of Oxford, Andrew Wiles Building, Woodstock Road, Oxford OX2 6GG, UK.}
\address[Sengul-address]{Faculty of Engineering and Natural Sciences, Sabanci University, Tuzla 34956, Istanbul, Turkey.}

\begin{abstract}
We consider a system of evolutionary equations that is capable of describing certain viscoelastic effects in linearized yet nonlinear models of solid mechanics. The essence of the paper is that the constitutive relation, involving the Cauchy stress, the small strain tensor and the symmetric velocity gradient, is given in an implicit form. For a large class of implicit constitutive relations we establish the existence and uniqueness of a global-in-time large-data weak solution. We then focus on the class of so-called limiting strain models, i.e., models for which the magnitude of the strain tensor is known to remain small a~priori, regardless of the magnitude of the Cauchy stress tensor. For this class of models, a new technical difficulty arises, which is that the Cauchy stress is only an integrable function
over its domain of definition, resulting in the underlying function spaces being nonreflexive and thus the weak compactness of bounded sequences of elements of these spaces is lost. Nevertheless, even for problems of this type we are able to provide a satisfactory existence theory, as long as the initial data have finite elastic energy and the boundary data fulfill natural compatibility conditions.
 \end{abstract}

\begin{keyword}
nonlinear viscoelasticity\sep strain-limiting theory \sep evolutionary problem \sep global existence \sep weak solution \sep regularity \MSC{35M13\sep 35K99 \sep 74D10 \sep 74H20}
\end{keyword}

\end{frontmatter}

\setcounter{equation}{0}
\numberwithin{equation}{section}

\section{Introduction}
%\noindent

%\subsection{The model}\label{section:model}
%\noindent
This paper is devoted to the study of the following nonlinear system of partial differential equations (PDEs). We assume that  $\Omega \subset \mathbb{R}^{d}$ is a given bounded open domain and we denote the parabolic cylinder by $Q:=(0,T)\times \Omega$ and its spatial boundary by $\Gamma:=(0,T)\times \partial \Omega$, where $T>0$ is the length of the time interval of interest. For given data $\bG:\R^{d\times d}_{sym}\to \R^{d\times d}_{sym}$, $\bef:Q\to \R^d$, $\bu_I:\Omega\to \R^d$, $\bv_0:\Omega\to \R^d$, $\bu_{\Gamma}: \Gamma \to \R^d$ and \(\alpha\), $ \beta>0$,  we seek a couple  $(\bu, \bT):Q \to \R^d \times \R^{d\times d}_{sym}$ satisfying
\begin{subequations}\label{model}
\begin{alignat}{2}
 \partial^2_{tt}\bu - \diver \bT & = \bef &&\quad \textrm{in $Q$},\label{linear-moment}\\
\label{cons-law}
\alpha \beps(\bu) + \beta \beps(\partial_t\bu) &= \bG(\bT)&&\quad \textrm{in $Q$},\\
\bu(0)  = \bu_{I},\quad \partial_t\bu(0) & = \bv_{0} &&\quad \textrm{in $\Omega$},\label{init-data}\\
\bu & = \bu_{\Gamma} &&\quad \text{on $\Gamma$}.\label{bound-data}
\end{alignat}
\end{subequations}
Here, \eqref{linear-moment} represents an approximation\footnote{In fact, the density $\varrho$ of the solid should also appear in \eqref{linear-moment}. In principle $\varrho$ could be a function of space and time and should satisfy the balance of mass equation. Since we are dealing with small strains here, i.e. the case when the deformation of the solid is small, under the assumption that the solid is homogeneous at initial time $t=0$, we can consider the density to be equal to a constant for all times $t\in(0,T)$. We shall therefore scale the density to be identically equal to one for simplicity; see also the discussion in~\cite{BuMaRa12}. We note however that under suitable assumptions it is not too difficult to extend the results presented herein to the case of variable density.} of the balance of linear momentum, where $\bef$ is the density of the external body forces, $\bu$ is the displacement, $\bT$ denotes the Cauchy stress tensor and the operator $\diver$ denotes the standard divergence operator with respect to the spatial variables $x_1,\ldots, x_d$.
%For simplicity, we do not consider the influence of a variable density in (\ref{linear-moment}) and without loss of generality we may assume that the constant density is equal to 1. However, we note that it is not too difficult to adapt the results of this paper to include a variable density under suitable assumptions. 
The Cauchy stress tensor $\bT$ is implicitly related to the small strain tensor $\beps(\bu):=\frac12 (\nabla \bu + (\nabla \bu)^{\rm T})$ and to the symmetric velocity gradient $\beps(\partial_t\bu):= \partial_t(\beps(\bu))$ via \eqref{cons-law}. The initial displacement and the initial velocity are given by \eqref{init-data} and the Dirichlet boundary condition for the displacement is represented by \eqref{bound-data}. A more detailed discussion concerning the relevance of \eqref{model} to problems in viscoelasticity is contained in Section~\ref{physics}.

It remains to specify the form of the implicit constitutive law \eqref{cons-law}. The minimal assumptions imposed on the mapping $\bG$ throughout the paper are the following. We assume that the function $\bG:\Rsym \to \Rsym$ is a continuous mapping such that, for some $p\in [1,\infty)$, some positive constants $C_1$ and $C_2$, and for all \(\bT\), $\bW \in \Rsym$, the following inequalities hold:
\begin{align}
\tag{A1} \big(\bG(\bT) - \bG(\bW)\big) \cdot (\bT - \bW) &\ge 0,\label{A1}\\ %(\text{MONOTONICITY})
\tag{A2} \bG(\bT) \cdot \bT &\geq C_{1} |\bT|^{p} - C_{2},\label{A2}\\
\tag{A3} |\bG(\bT)| &\leq C_{2} (1 + |\bT|)^{p-1},\label{A3}
\end{align}
where $| \cdot |$ stands for the usual Frobenius matrix norm.
Assumptions \eqref{A1}--\eqref{A3} are sufficient for the existence and uniqueness of a weak solution provided that $p\in (1,\infty)$. For $p =1$, however, we must impose a more restrictive assumption due to the lack of compactness experienced when working in \( L^1\). Namely, we will assume that there exists a strictly convex function $\phi \in \mathcal{C}^2(\R_+; \R_+)$ such that $\phi(0)=\phi'(0)=0$, $|\phi''(s)|\le C(1+s)^{-1}$ for every \(s\in \mathbb{R}_+\), and for all $\bT\in \Rsym$ there holds
\begin{align}\tag{A4}\label{A4}
\bG(\bT)=\frac{\phi'(|\bT|) \bT}{|\bT|}.
\end{align}

In order to simplify the exposition and avoid nonessential technical details concerning the choice of appropriate function spaces that admit suitable trace theorems, we shall assume that there exists a function $\bu_0:Q \to \R^d$ fulfilling, in an appropriate sense, the initial and boundary conditions
\[
\begin{aligned}
\bu_0(0)&=\bu_I &&\text{in }\Omega,\\
\partial_t \bu_0(0)&=\bv_0 &&\text{in } \Omega,\\
\bu_0&=\bu_{\Gamma} &&\text{on } \Gamma.
\end{aligned}
\]
We shall henceforth formulate  all assumptions on the initial and boundary data in terms of $\bu_0$, rather than $\bu_I$, $\bv_0$ and $\bu_\Gamma$. We note that since the function spaces for $\bu_0$ stated below are the same as those for the weak solution $\bu$,  it is in fact necessary that such a $\bu_0$ exists. Otherwise our construction of a weak solution would not be possible.

\subsection{Statement of the main results}\label{mathematics}

First, we formulate our result for the case when $p>1$.
\begin{theorem}\label{T1}
Let  $p'>2d/(d+2)$, let $\bG$ satisfy \eqref{A1}, \eqref{A2} and \eqref{A3}, and let  \(\alpha\), $ \beta > 0$ be arbitrary. Assume that the data satisfy the following hypotheses:
\begin{equation}
\begin{split}
\bu_0&\in W^{1,p'}(0,T; W^{1,p'}(\Omega;\R^d))\cap W^{2,p}(0,T; (W^{1,p'}_0(\Omega;\R^d))^*)\cap \mathcal{C}^1([0,T];L^2(\Omega;\R^d)),\\
\bef&\in L^{p}(0,T; (W^{1,p'}_0(\Omega;\R^d))^*).
\end{split}\label{data-as}
\end{equation}
Then, there exists a  couple $(\bu, \bT)$ fulfilling
\begin{align}
\bu &\in \mathcal{C}^{1}([0, T]; L^{2}(\Omega;\R^d))\cap W^{1, p'}(0, T; W^{1, p'}(\Omega;\R^d))\cap  W^{2,{p}}(0, T; (W_{0}^{1,p' }(\Omega;\R^d))^{*}), \label{FSu}\\
\bT &\in L^{p}(0, T; L^{p}(\Omega;\Rsym))\label{FST}
\end{align}
and solving \eqref{model} in the following sense:
\begin{align}\label{WF}
\langle \partial_{tt}\bu, \bw \rangle + \int_{\Omega} \bT \cdot \nabla \bw &= \langle \bef, \bw \rangle && \forall\, \bw \in W_{0}^{1,p'}(\Omega;\R^d),\quad
\textrm{for a.e. }\,t \in (0, T),
\\
\label{T-const}
\alpha \beps(\bu) + \beta \partial_{t} \beps(\bu)&=\bG(\bT) &&\textrm{a.e. in }Q,
\end{align}
and
\begin{equation}
\bu-\bu_0 =\b0 \quad \textrm{ a.e. on }\Gamma\qquad \textrm{and}\qquad  \bu(0)-\bu_0(0)=\partial_t \bu(0)-\partial_t\bu_0(0)=\b0\quad \textrm{ a.e. in }\Omega. \label{bcint}
\end{equation}
Furthermore, the function \( \bu \) is unique. If, additionally, the mapping \( \bG \) is strictly monotonic, then \( \bT \) is also unique.
\end{theorem}
Before proceeding, we will first comment on the assertions of Theorem~\ref{T1}. The proof of Theorem~\ref{T1} is based on the relevant a~priori estimates, for which the assumption \eqref{data-as} seems to be both optimal and  minimal. The function spaces considered  in \eqref{FSu}, \eqref{FST} correspond to the structural assumptions imposed on $\bG$, namely the coercivity assumption \eqref{A2} and the growth condition \eqref{A3}. Since $p>1$, we have a ``standard" function space setting, so the nonlinearity in \eqref{T-const} can be identified by using a modification of Minty's method. Theorem~\ref{T1} can also be understood as an extension of the results established in ~\cite{BuMaRa12}; similarly as here, the authors of~\cite{BuMaRa12} treated a viscoelastic solid model of generalized Kelvin--Voigt type, but they considered a constitutive relation for the Cauchy stress of the following explicit form:
\begin{align*}
\bT&=\bT_{el}(\beps(\bu)) + \bT_{vis}(\partial_t\beps(\bu)) \qquad \textrm{a.e. in }Q.
\end{align*}
The regularity results for such models are available in~\cite{BuKaSt13}. It is remarkable that while \eqref{T-const} can be fully justified from the physical point of view via implicit constitutive theory, see \cite{Raj-03}, the above explicit form $\bT=\bT_{el} + \bT_{vis}$ can be justified for particular choices of $\bT_{el}$ and $\bT_{vis}$ only.

In contrast with the case of $p>1$, almost none of what was said above applies in the case $p=1$, or for the limit, as $p\to 1_+$, of the sequence of solutions constructed in Theorem~\ref{T1}. Indeed, for similar models in the purely steady elastic setting, it was demonstrated  in \cite{BeBuMaSu17} that $\bT$ is, in general, a Radon measure and therefore one can hardly consider \eqref{T-const} pointwise in $Q$. Nevertheless, it was  shown there that under some structural assumptions on $\bG$ (corresponding to \eqref{A4}), one may hope for $\bT$ to be integrable. A similar situation was also studied in~\cite{BeBuGm20} but with $p\to \infty$, which, in general, leads to solutions $\bu$ in $BV$ spaces. However,  under a structural assumption related to~\eqref{A4}, one can again overcome such difficulties and show the existence of a solution that belongs to a Sobolev space.

A similar situation can be expected in our setting  when $p=1$. Therefore, in order to avoid difficulties associated with the interpretation of $\partial_{tt}\bu$ and the interpretation of the sense in which the initial data are attained, we assume here, for simplicity, that the right-hand side $\bef\in L^2(Q;\mathbb{R}^d)$. %\textcolor{red}{We also use a variational formulation which is slightly different from \eqref{WF}. Nevertheless, we will show that \eqref{WF} still holds locally in $(0,T)$ and, in the case of more regular initial data, we are able to show the continuity with respect to time of $\bu$ and $\partial_t \bu$ on the whole time interval $[0,T]$. {\bf (VP: I would like to delete these lines. I think that they are slightly misleading as the form of the weak problem is essentially the same?)}}

Thus, inspired by \cite{BeBuMaSu17}, if \( p =1\) we assume in addition to \eqref{A1}--\eqref{A3}  that we have \eqref{A4}.
%
%
%
%
%For the limiting strain case, i.e., for $p=1$, we shall assume that in addition to \eqref{A1}--\eqref{A3}, also the structural assumption \eqref{A4} holds true with a \emph{strictly} convex $\phi\in \mathcal{C}^2(\R_+; \R_+)$. For simplicity, we also require that
%\begin{equation}\label{at0}
%\phi(0)=\phi'(0)=0.
%\end{equation}
It then follows from the structural assumptions that for all $s\in \R_+$ we have
\begin{equation*}%\label{atinfty}
\begin{split}
\frac{C_1 s}{2} - C_2 &\le \phi(s)\le C_2 s,\\
0&\le \phi'(s) \le C_2.
\end{split}
\end{equation*}
Since $\phi$ is convex, we deduce that there exists an $L>0$ such that
\begin{equation}\label{dfL}
L:=\lim_{s\to \infty} \phi'(s) \ge \phi'(t) \qquad \forall \,t\in \R.
\end{equation}
The number $L$ plays an essential role in the subsequent analysis, in particular in the assumptions on the initial and boundary data. Indeed, thanks to \eqref{A4}, we see that
\begin{equation}\label{dfL2}
L=\lim_{|\bW|\to \infty} |\bG(\bW)|\ge |\bG(\bT)| \qquad \forall \,\bT\in \Rsym.
\end{equation}
Hence, if \eqref{cons-law} is satisfied then we must necessarily have
\begin{equation}\label{nec}
|\alpha \beps(\bu) + \beta \partial_t \beps(\bu)|\le L \qquad \textrm{ a.e. in }Q.
\end{equation}
Consequently, if such a $\bu$ should exist then it is natural to assume the same requirement as \eqref{nec} also for the initial and boundary data, that is, we must have
\begin{equation}\label{nec0}
|\alpha \beps(\bu_0) + \beta \partial_t \beps(\bu_0)|\le L \qquad \textrm{ a.e. in }Q.
\end{equation}
In fact, we require in the existence analysis that \eqref{nec0} is satisfied with a strict inequality sign; such a condition is called the {\it  safety strain condition}.

%%%%%%%%%%%%%%%%%%%%%%%%%%%%%%%%%%%%%%%%%%%%%%%%%%%%%%%%%%%%%%%%%%%%%%%%Figures
\tikzstyle{dot}=[draw,fill=white,circle,inner sep=0pt,minimum size=4pt]
\def\deltazero{
\begin{tikzpicture}[scale=2]
    \footnotesize
    \draw[ultra thin] (-0.2,  0.0) -- (3.2, 0.0) node[below]{$s$};
    \draw[ultra thin] ( 0.0, -0.2) -- (0.0, 3.2) node[left ]{$\phi(s)$};
    \draw[ultra thin] ( -0.05, 1) -- (0.05, 1);
      %\node at (0, 1) {L};
      %\node[label={below:1}] at (1, 0) {};
       \draw[variable=\t,domain=0:3.2,thick,samples=400]
      plot (\t         ,{\t - ln(\t +1)}) ;
       \draw[variable=\t,domain=0:3.2,thick,dashed,samples=400]
      plot ({\t}         ,{(1+(\t)^(2))^(1/2)-1} ) ;
 \end{tikzpicture}
}
\def\deltaone{
\begin{tikzpicture}[scale=1.5]
    \footnotesize
    \draw[ultra thin] (-0.2,  0.0) -- (2.2, 0.0) node[below]{$\abs{\bG(\bT)}$};
    \draw[ultra thin] ( 0.0, -0.2) -- (0.0, 2.2) node[left ]{$|\bT|$};
   % \draw[ultra thin] ( -0.05, 1) -- (0.05, 1);
    %\draw[variable=\t,domain=0:3.2,dotted,samples=400]
     % plot ({\t}         ,{1}) (-0.1,1) node{L} ;
      %\node at (0, 1) {L};
      %\node[label={below:1}] at (1, 0) {};
      \draw[variable=\t,domain=0:2.2,dotted,samples=400]
      plot ({\t}         ,{1}) (-0.1,1) node{1} ;
 \draw[variable=\t,domain=0:2,thick,dashdotted,samples=400]
      plot ({\t},{\t/(pow(1+pow(\t,10),1/10))}) ;
        \draw[variable=\t,domain=0:2,thick,samples=400]
      plot ({\t}         ,{(\t)/(1+\t)} ) ;
       \draw[variable=\t,domain=0:2,thick,dashed,samples=400]
      plot ({\t}         ,{(\t)/((1+(\t)^(2))^(1/2))} ) ;
  \end{tikzpicture}
}
%%%%%%%%%%%%%%%%%%%%%%%%%%%%%%%%%%%%%%%%%%%%%%%%%%%%%%%%%%%%%%%%%%%%%%%%%%%%%%%%%%%%%%%%%%%%%%%%%%%%%%%%%%%%%%%

\begin{theorem}\label{T4.1}
For some strictly convex $\phi\in \mathcal{C}^2(\R_+; \R_+)$, let $\bG$ satisfy \eqref{A1}--\eqref{A4} with $p=1$. Assume that the data satisfy the following hypotheses:
\begin{equation}
\begin{split}
\bu_0&\in W^{1,\infty}(0,T; W^{1,2}(\Omega;\R^d))\cap W^{2,1}(0,T;L^2(\Omega;\R^d)),\\
\bef&\in L^{2}(0,T; L^2(\Omega;\R^d)),
\end{split}\label{data-as2}
\end{equation}
with the safety strain condition
\begin{equation}\label{compt}
\|\alpha \beps(\bu_0) + \beta\partial_t \beps(\bu_0)\|_{L^{\infty}(Q; \Rsym)}< L
\end{equation}
and the following bound holds for every \( \delta>0 \):
\begin{equation}\label{compt2}
\esssup_{(t,x)\in (\delta,T)\times \Omega} |\partial_{tt} \beps (\bu_0(t,x))|< \infty.
\end{equation}
Then, there exists a unique couple $(\bu,\bT)$ fulfilling
\begin{align}
\bu &\in  W^{1,\infty}(0, T; W^{1, 2}(\Omega;\R^d))\cap \mathcal{C}^1([0,T]; L^2(\Omega;\mathbb{R}^d))\cap  W^{2,{2}}_{loc}(0, T; L^2(\Omega;\R^d) ), \label{FSu2}\\
\beps(\bu) &\in  L^{\infty}(Q;\Rsym), \label{FSe2}\\
\partial_t\beps(\bu) &\in  L^{\infty}(Q;\Rsym), \label{FSe22}\\
\bT &\in L^{1}(0, T; L^{1}(\Omega;\Rsym))\label{FST2}
\end{align}
and satisfying
\begin{align}\label{WF2}
\int_{\Omega}\partial_{tt}\bu \cdot \bw + \bT \cdot \nabla \bw\dx &= \int_{\Omega}\bef \cdot \bw \dx &&\forall \, \bw \in W_{0}^{1,\infty}(\Omega;\R^d),\quad \text{for a.e. }\,t \in (0, T),\\
\label{T-const2}
\alpha \beps(\bu) + \beta \partial_{t} \beps(\bu)&=\bG(\bT) &&\textrm{a.e. in }Q,
\end{align}
and
\begin{equation}
\bu-\bu_0 =\b0 \quad \textrm{ a.e. on }\Gamma\qquad \textrm{and}\qquad  \bu(0)-\bu_0(0)=\partial_t \bu(0)-\partial_t\bu_0(0)=\b0 \quad \textrm{ a.e. in }\Omega. \label{bcint2}
\end{equation}
\end{theorem}

%{\color{blue}

This theorem answers the question of existence of weak solutions to the problem under the assumptions \eqref{A1}--\eqref{A3} when $p\to 1_+$ and therefore provides an existence result for limiting strain models for which the symmetric displacement gradient and symmetric velocity gradient remain bounded; see Section~\ref{physics} for the physical background and the importance of the model. We note that a very similar existence result was established recently in~\cite{BuPaSeSu20}; there are however certain essential differences, which make the results of the present paper much stronger. First, in~\cite{BuPaSeSu20} the authors only consider the prototypical model
\begin{align}\label{A4p}
\bG(\bT):=\frac{\bT}{(1+|\bT|^q)^{\frac{1}{q}}},
\end{align}
while we are able to cover here a more general class of models under hypothesis \eqref{A4}. The corresponding potential $\phi$ (whose existence is assumed in \eqref{A4}) is, for the model \eqref{A4p}, given by
$$
\phi(s):=\int_0^s \frac{t}{(1+t^q)^{\frac{1}{q}}}\dt,\qquad s \in \mathbb{R}_+.
$$
The role of the parameter $q$ in \eqref{A4p} is indicated in Fig.~\ref{Fig1}.
\begin{figure}[h]
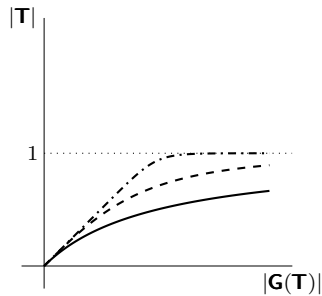

    \centering
    %\deltazero{}
     \deltaone{}
    \caption{Dependence of $|\bG|$ on $|\bT|$ for the prototype model~\eqref{A4p}. The three curves correspond to $q=1$ (solid curve), $q=2$ (dashed curve) and $q=10$ (dash-dotted curve). Clearly, $|\bT|$ tends to $1$ more rapidly with increasing~$q$.}
    \label{Fig1}
\end{figure}
Second, the paper~\cite{BuPaSeSu20} is concerned with the spatially periodic setting, which simplifies the analysis in an essential way, most notably with regard to the derivation of the relevant a~priori estimates. Finally, in~\cite{BuPaSeSu20} the initial data are assumed to be quite regular (they are supposed to belong to the Sobolev space $W^{k,2}(\Omega;\mathbb{R}^d)$ with $k>\frac{d}{2}$), which is related to the choice of the method used therein to prove the existence of a weak solution. In this paper we do not require such strong regularity of the initial data. Nevertheless, since in our context here it is difficult to describe the correct space-time trace spaces, because we are dealing with $L^{\infty}$-type spaces and symmetric gradients, and since we want to state the result in its full generality (so as to be able to admit time-dependent boundary data), we do assume a certain compatibility condition via an a~priori prescribed space-time function $\bu_0$ that we use in order to impose the initial and boundary conditions. Indeed, the existence of $\bu_0$ satisfying the safety strain condition \eqref{compt} is necessary for the existence of a solution. Next, the assumption \eqref{data-as2}$_1$, requiring temporal regularity of $\bu_0$,  is required in order to ensure that $\bu_0$ and $\partial_t \bu_0$ have meaningful traces at time $t=0$. Finally, the assumption~\eqref{compt2} prescribes the required temporal smoothness of the boundary data, but since it only involves $t \in (\delta, T)$ with $\delta>0$, it has nothing to do with either the regularity of the initial condition or its compatibility with the boundary data.   We give several examples for simplified settings in the following remark.
\begin{remark}
%\normalfont
We discuss two cases of boundary and initial data from \eqref{init-data}--\eqref{bound-data} for which it is easy to construct a function $\bu_0$ that satisfies the assumptions \eqref{data-as2}--\eqref{compt2}. %Also it will be evident, compare it also with \eqref{energy-2} and the computation for the free energy, that the natural assumption is $\alpha \|\beps(\bu_I)\|_{\infty}<L$.

\bigskip

\noindent \textsf{Boundary data independent of time.} Suppose that $\bu_{\Gamma}$ is independent of time and $\bu_I\in W^{1,2}(\Omega;\R^d)$ satisfies the compatibility condition $\bu_I|_{\partial \Omega}=\bu_{\Gamma}$. Because the
boundary data are independent of time, it is natural to assume that $\bv_0 \in W^{1,2}_0(\Omega; \R^d)$, where
\begin{equation}\label{case1}
\|\alpha \beps(\bu_I) +\beta \beps(\bv_0)\|_{L^\infty(\Omega;\mathbb{R}^{d\times d}_{sym})} <L.
%\alpha \|\beps(\bu_I)\|_{\infty} \quad \textrm{ and } \quad
\end{equation}
Then we can set
$$
\bu_0(t,x):=\mathrm{e}^{-\frac{\alpha t}{\beta}}\bu_I(x) +\frac{\alpha \bu_I(x) + \beta \bv_0(x)}{\alpha}(1-\mathrm{e}^{-\frac{\alpha t}{\beta}}).
$$
Indeed, a direct computation yields that
$$
\partial_t  \bu_0(t,x)= \bv_0(x)\,\mathrm{e}^{-\frac{\alpha t}{\beta}},
$$
and thus, $\bu_0(0,x)=\bu_I(x)$, $\partial_t \bu_0(0,x)=\bv_0(x)$ for $x \in \Omega$ and $\bu_0|_{\Gamma} = \bu_{\Gamma}$. Moreover,
$$
\alpha \beps(\bu_0) + \beta \partial_t \beps(\bu_0)= \alpha \beps(\bu_I) + \beta \beps(\bv_0).
$$
Consequently, $\bu_0$ satisfies \eqref{compt} provided \eqref{case1} holds. The validity of \eqref{compt2} is obvious.

\bigskip

\noindent \textsf{Time-dependent boundary data.} In this setting, we a~priori assume the existence of some $\tilde{\bu}$ such that $\tilde{\bu}(0,x)=\bu_I(x)$ for $x \in \Omega$ and $\tilde{\bu}|_{\Gamma}=\bu_{\Gamma}$. In addition, it is natural to assume the compatibility condition $\bv_0(\cdot)=\partial_t \bu_{\Gamma}(0,\cdot)$ on $\partial \Omega$. We adopt the following assumption on $\tilde{\bu}$ and $\bv_0$:
\begin{equation}\label{case2}
\|\alpha \beps(\tilde{\bu})+ \beta(\partial_t\beps(\tilde{\bu})-\partial_t \beps(\tilde{\bu}(0,\cdot)) +\beps(\bv_0(\cdot)))\|_{L^\infty(Q;\mathbb{R}^{d \times d}_{sym})} <L.
\end{equation}
We define
$$
\bu_0(t,x):= \tilde{\bu}(t,x)+\frac{\beta (\bv_0(x)-\partial_t \tilde{\bu}(0,x))}{\alpha}(1-\mathrm{e}^{-\frac{\alpha t}{\beta}}).
$$
Clearly, $\bu_0(0,x)=\tilde{\bu}(0,x)=\bu_I(x)$ for $x \in \Omega$ and $\bu_0=\bu_{\Gamma}$ on $\Gamma$. The time derivative of $\bu_0$ is
$$
\partial_t \bu_0(t,x)= \partial_t\tilde{\bu}(t,x)+(\bv_0(x)-\partial_t \tilde{\bu}(0,x))\,\mathrm{e}^{-\frac{\alpha t}{\beta}}
$$
and thus $\partial_t \bu_0(0,x)= \bv_0(x)$ for $x \in \Omega$. In addition, since
$$
\alpha \beps(\bu_0) + \beta \partial_t \beps (\bu_0)=\alpha \beps(\bu_I)+ \beta(\partial_t\beps(\bu_I)-\partial_t \beps(\bu_I(0)) +\beps(\bv_0)),
$$
we see that \eqref{compt} is equivalent to \eqref{case2}. The assumption \eqref{compt2} is then only related to our extension of the boundary data inside of $\Omega$ and the temporal regularity of the boundary data.

\end{remark}

%}

\subsection{Relevance to the modelling of viscoelastic solids}\label{physics}

With these results in mind, we will now discuss the importance of such problems.
We often encounter materials exhibiting viscoelastic response. By definition, viscoelasticity involves the material response of both elastic solids and viscous fluids, which can be modelled linearly or nonlinearly (see \cite{Sengul-visco-review} for an extensive overview). On the other hand, it is well-known that implicit constitutive theories allow for a more general structure in modelling than explicit ones (cf. \cite{Raj-03}, \cite{Raj-06}), where the strain could be given as a function of the stress. Indeed, this is the case in our constitutive relation \eqref{cons-law} in system \eqref{model}. Rajagopal's main contribution \cite{Raj-10} to the theory was to show that a nonlinear relationship between the stress and the strain can be obtained after linearizing the strain. The relation \eqref{cons-law} was first obtained by Erbay and \c{S}eng\"{u}l in \cite{Erbay-Sengul}  as a result of the linearization of the relation between the stress and the strain tensors under the assumption that the magnitude of the strain is small.  For models of this type it is possible that once the magnitude of the strain has reached a certain limiting value (as is the case in Theorem \ref{T4.1}), any further increase of the magnitude of the stress will cause no changes in the strain. These models are called \textit{strain-limiting (strain-locking) models} and such  behaviour has been observed in numerous experiments (see \cite{Sengul-strain-lim-review} and references therein). For a further discussion of such models in the purely elastic setting or in the setting of the generalized Kelvin--Voigt model  we refer to \cite{BuMaRa12}, and in the viscoelastic setting to \cite{Erbay-Sengul, Sengul-strain-lim-review}.

Now we introduce some basic kinematics in order to discuss these limiting strain models from a mathematical perspective.
We denote by $\bu(\mathbf{X}, t):=\bx(\mathbf{X},t)-\mathbf{X}$ the displacement of a given  body at a space-time point $(\mathbf{X},t)$, where $\mathbf{X}$ is the position vector in the reference configuration and
$\bx(\mathbf{X},t)$ is the position vector in the current configuration. We denote the deformation of the body, which is assumed to be stress-free initially, by $\boldsymbol{\chi}(\mathbf{X}, t)$. The deformation gradient
%, which is the mathematical tool used to describe choices of functions as models of deformation,
is defined as $\mathbf{F} = \partial \boldsymbol{\chi}/ \partial \mathbf{X} $. By the polar decomposition theorem, we can ensure the existence of positive definite, symmetric tensors $\bU$, $\bV$, and a rotation $\bR$ such that
\[\bF= \bR \bU = \bV \bR,\]
where $\bU$ and $\bV$ are the \textit{right} and \textit{left Cauchy--Green stretch tensor}, respectively.
Moreover, we know that each of these decompositions is unique and
\[\bC = {\bU}^{2} = {\bF}^{\rm T} \bF, \quad \bB  = {\bV}^{2} = \bF {\bF}^{\rm T},\]
where $\bB$, $\bC$ are called the \textit{right} and \textit{left Cauchy--Green deformation tensor}, respectively. We define the velocity as $\bv = \partial \boldsymbol{\chi}/ \partial t$ and denote by $\bD$ the symmetric part of the gradient of the velocity field $\bL = \partial \bv / \partial \bx $.
%That is, \[\mathbf{D} = \frac{1}{2} (\mathbf{L} + \mathbf{L}^{T}).\]
Under the assumption that
\begin{equation}\label{smallness}
\| \nabla \bu \|_{L^\infty(Q;\mathbb{R}^{d\times d})} = O(\delta), \qquad 0<\delta \ll 1,
\end{equation}
one can obtain the linearized strain, mentioned previously, as
\begin{equation}\label{lin-strain}
\beps (\bu)= \frac{1}{2} \left[\nabla \bu + (\nabla \bu)^{\rm T}\right].
\end{equation}
%We will next consider constitutive equations relating \( \bT\), \( \bB\) and \( \bD\). Under the smallness assumption (\ref{smallness}) we have that \( |\bB - (\bI  + \beps)|= O(\delta^2) \). Thus we obtain a constitutive relation between \( \bT \), \( \beps\) and \( \beps_t\).
%Using this definition, one can define the linearized counterpart of $\mathbf{D}$ as $\boldsymbol{\epsilon}_{t} = \partial \boldsymbol{\epsilon}/\partial t$.

As is explained in \cite{Sengul-strain-lim-review}, in the purely elastic setting, starting from the following constitutive relation between the stress and the strain
\begin{equation}\label{implicit}
\bG(\bT, \bB) = \mathbf{0},
\end{equation}
for frame-indifferent and isotropic bodies, one can obtain the representation
\begin{equation}\label{representation}
\begin{split}
\bG(\bT, \bB) & = \chi_{0} \bI + \chi_{1} \bT + \chi_{2} \bT + \chi_{3} {\bT}^{2} + \chi_{4} \bB^{2} + \chi_{5} (\bT \bB + \bB \bT) \\
& \quad + \chi_{6} (\bT^{2} \bB + \bB \bT^{2}) + \chi_{7} (\bB^{2} \bT + \bT \bB^{2}) + \chi_{8}  (\bT^{2} \bB^{2} + \bB^{2} \bT^{2}),
\end{split}
\end{equation}
where the functions $\chi_i$, $i=0,\dots,8$,  depend only on the scalar invariants of $\bT$ and $\bB$, which can be expressed in terms of
\[\mathrm{tr}\,{\bT}, \mathrm{tr}\,{\bB}, \mathrm{tr}\,{\bT^{2}}, \mathrm{tr}\,{\bB^{2}}, \mathrm{tr}\,{\bT^{3}}, \mathrm{tr}\,{\bB^{3}}, \mathrm{tr}\,{\bT \bB}, \mathrm{tr}\,{\bT^{2}} \bB, \mathrm{tr}\,{\bT \bB^{2}}, \mathrm{tr}\,{\bT^{2} \bB^{2}}.\]
Under the smallness assumption (\ref{smallness}), we have that \( |\bB - (\bI  + \beps)|= O(\delta^2) \), with $\beps=\beps(\bu)$. Thus, at the end of the linearization process, \eqref{representation} gives a nonlinear relationship between \( \bT \) and \( \beps\).
In many studies a simpler subclass of constitutive relations than \eqref{representation} is considered, namely
\begin{equation}\label{B-T}
\bB =  \tilde{\chi}_{0} \bI + \tilde{\chi}_{1} \bT + \tilde{\chi}_{2} \bT^{2}.
\end{equation}
Under the assumption \eqref{smallness}, the equality \eqref{B-T} becomes
\begin{equation}\label{quad-T}
\beps  = \bar{\chi}_{0} \bI+ \bar{\chi}_{1} \bT + \bar{\chi}_{2} \bT^{2},
\end{equation}
with some invariant-dependent coefficients $\bar{\chi}_i$, $i=0,1,2$. The analysis of a limiting strain problem with a constitutive relation of the form \( \beps = \bG(\bT) \), which is a more general version of \eqref{quad-T}, with a bounded mapping \( \bG \), as those considered here, was also studied in \citep{BuMaRaSu}, \cite{BeBuMaSu17}, where the authors highlight the analytical difficulties associated with such models, most notably the lack of weak compactness of approximations to the stress tensor in \( L^1(\Omega;\mathbb{R}^{d \times d}_{sym})\). We rely on methods developed in \cite{BeBuMaSu17} in order to show that (\ref{WF2}) holds for our proposed solution of the problem. The additional time-dependence here presents further difficulties in the analysis. In particular, we must develop suitable space-time estimates.

Next we focus on the derivation of the constitutive relation of interest in this paper.
In the viscoelastic setting, as is explained in \cite{Erbay-Sengul}, instead of \eqref{implicit} one would start with a general implicit constitutive relation of the form
\begin{equation}\label{implicit-visco}
\bG(\bT, \bB, \bD) = \mathbf{0}.
\end{equation}
For simplicity and in view of (\ref{B-T}), we study the following subclass of such models:
\begin{equation}\label{visco-cons}
\alpha \bB + \beta \bD = \gamma_{0} \bI + \gamma_{1} \bT + \gamma_{2} \bT^{2},
\end{equation}
where $\gamma_{i} = \gamma_{i}(I_{1}, I_{2}, I_{3})$, $i = 0, 1, 2,$ $I_{1} = \text{tr} \bT, I_{2} = \frac{1}{2} \text{tr} \bT^{2}, I_{3} = \frac{1}{3} \text{tr} \bT^{3},$ and  $\alpha$, $\beta$ are nonnegative constants. We note that under assumption (\ref{smallness}) we may interchange derivatives with respect to \( \bx \) and \( \mathbf{X}\). In particular, the linearized counterpart of \( \mathbf{D}\) can be identified with \( \beps_t = \beps(\bu_t) \). Therefore, assuming \eqref{smallness} and writing the right-hand side of \eqref{visco-cons} more generally as a nonlinear function of $\bT$, one obtains \eqref{cons-law} as required.

Models of the type (\ref{visco-cons}) were considered in \cite{RaSa14} in order to describe viscoelastic solid bodies. The model is a generalization of the classical (linear) Kelvin--Voigt model, which in one space dimension involves the constitutive relation
\begin{equation}\label{K-V-linear}
\sigma = E\epsilon + \eta\epsilon_t,
\end{equation}
where \(\sigma\) denotes the scalar stress, \( \epsilon\) the scalar strain, and \( E\), \( \eta\) are constants signifying the modulus of elasticity and the viscosity, respectively.  As mentioned before, it is worth noting that similar models have been considered in \cite{BuMaRa12,BuKaSt13}, where the authors assumed that the stress \( \bT \) was a sum of the elastic \( \bT_{el} \) and viscous \( \bT_{vis}\) parts. Considering implicit relations for each component separately, they obtained \( \bT_{el} = \tens{H}(\beps)\), \( \bT_{vis} = \tens{G}(\beps_t)\) for nonlinear mappings \( \bH\), \( \bG\). However,  the assumptions that were made there on \( \tens{H}\) and \( \tens{G}\)  result in a problem that is no longer of strain-limiting type. This, together with the additive decomposition of the stress considered there, led to an analysis that is very different from the one performed here.

Some analysis (albeit limited) of the problem (\ref{model}) is available in the literature, which we now discuss. In one space dimension the authors of \cite{Erbay-Sengul} derived the equation
\begin{equation}\label{strain-rate}
\sigma_{xx} + \beta \sigma_{xxt} = g(\sigma)_{tt},
\end{equation}
using the equation of motion \eqref{linear-moment} together with the constitutive relation \eqref{cons-law} and setting $\alpha = 1$, where, as in (\ref{K-V-linear}), \( \sigma\) refers to the scalar stress.  In (\ref{strain-rate}), the nonlinearity $g$ corresponds to $\bG$ in the current case.
The authors investigated conditions on the function $g$ under which travelling wave solutions exist.  Furthermore, in \cite{ErErSe} the authors proved the local-in-time existence of solutions for equation \eqref{strain-rate}.
In this work, we use the same set of equations without deriving a single equation on account of  the fact that we are working in a higher-dimensional setting. In particular, the symmetric gradient does not reduce to a classical gradient operator as in the one-dimensional case, a property that is exploited in \cite{Erbay-Sengul} and \cite{ErErSe}.

 A related problem is studied in \cite{ErSe20} where the authors looked at the stress-rate case instead of the strain-rate case. In the one-dimensional setting, this resulted in the equation
\begin{equation}\label{stress-rate}
\sigma_{xx} + \gamma \sigma_{ttt} = h(\sigma)_{tt}.
\end{equation}
The constitutive law for that study was $\epsilon + \gamma \sigma_{t}  = h(\sigma)$ instead of \eqref{cons-law}. The authors pointed out that travelling wave solutions of equations \eqref{strain-rate} and \eqref{stress-rate} will coincide. However, we do not attempt to explore the stress-rate problem here.

We close this section with a thermodynamical justification of the model (\ref{model}).
We will show in particular that the total energy of the system is constant and the sum of the kinetic energy and the elastic energy is a decreasing function of time.
We suppose that  the constitutive relation can be written as
\begin{align*}
\beps + \beta\partial_t\beps = \frac{\partial \varphi}{\partial\bT}(\bT) =: \bG(\bT)
\end{align*}
where \( \varphi\) is a function from \( \mathbb{R}^{d\times d}\) to \( \mathbb{R}_+\) defined by \( \varphi(\bT) = \phi(|\bT|)\) and $\beps=\beps(\bu)$.
We shall suppose that \( \phi({0}) = \phi'(0) = 0\) and assume that \( \phi \in \mathcal{C}^2(\mathbb{R}_+;\mathbb{R}_+) \) is strictly convex.
Clearly this is the case if (\ref{A4}) holds.
Under these assumptions, \( \varphi\) is also strictly convex, noting that \( \phi \) is strictly increasing on \( [0, \infty ) \). Furthermore \( \bG\) is monotone. Next, we  define the convex conjugate \( \varphi^*\) by
\[
\varphi^*(\beps) = \sup_{\bT \in \mathbb{R}^{d\times d}_{sym}} \big( \, \beps\cdot  \bT - \varphi(\bT)\big).
\]
We note that \( \varphi^*\) is also convex and, for any \( \bT \in \mathbb{R}^{d\times d}_{sym}\), the following identity holds:
\[
\varphi^*(\bG(\bT)) + \varphi(\bT) = \bG(\bT) \cdot  \bT.
\]
Thus, the function \( \bG^{-1} = \frac{\partial\varphi^*}{\partial\bT}\) is also monotone. With these facts in mind, formally testing  (\ref{linear-moment}) against \( \partial_t \bu \) and assuming the absence of all body forces, we obtain
\begin{equation}\label{energy-1}
\frac{1}{2}\frac{\mathrm{d}}{\mathrm{d}t} \int_{\Omega} |\partial_t \bu|^2 \,\mathrm{d}x + \int_\Omega \bT \cdot  \partial_t \beps(\bu) \,\mathrm{d}x = 0.
\end{equation}
However, the integrand in the second term on the right-hand side can be rewritten as
\begin{align*}
\bT \cdot  \partial_t \beps &= \frac{\partial\varphi^*}{\partial\bT}\cdot \partial_t \beps + \left( \bT - \frac{\partial\varphi^*}{\partial\bT}(\beps) \right)\cdot  \partial_t \beps
\\
&= \partial_t (\varphi^*(\beps)) + \frac{1}{\beta} \left( \bT - \frac{\partial\varphi^*}{\partial\bT}(\beps) \right)\cdot ( \bG(\bT) - \beps)
\\
&= \partial_t (\varphi^*(\beps)) + \frac{1}{\beta} \left( \bT - \bG^{-1}(\beps) \right)\cdot ( \bG(\bT) - \beps).
\end{align*}
Substituting this back into (\ref{energy-1}) and defining \( \bT_0 := \bG^{-1}(\beps)\), we see that
\begin{equation}\label{energy-2}
\frac{\mathrm{d}}{\mathrm{d}t }\left( \int_\Omega \frac{1}{2}|\partial_t \bu|^2 + \varphi^*(\beps) \,
\mathrm{d}x\right) + \frac{1}{\beta}\int_\Omega\left( \bT - \bT_0\right)\cdot ( \bG(\bT) - \bG(\bT_0))\,\mathrm{d}x =0.
\end{equation}
Recalling that \( \bG \) is monotone, we deduce that
\begin{align*}
\sup_{t\in (0,T)} \left(\int_\Omega \frac{1}{2}|\partial_t \bu|^2 + \varphi^*(\beps) \,
\mathrm{d}x\right) \leq \int_\Omega \frac{1}{2}|\bv_0|^2 + \varphi^*(\beps(\bu_I))\,\mathrm{d}x.
\end{align*}
In particular the sum of the kinetic energy and elastic energy is decreasing. The extra term that appears in (\ref{energy-2}) corresponds to the dissipation; thus energy is conserved in accordance with the laws of thermodynamics.

The structure of the remainder of the paper is as follows. In Section \ref{Thm1} we prove Theorem \ref{T1}. We structure the proof in the following way.  First, in Section \ref{galerkin} we use a Galerkin method and find a weak solution to an approximate problem. In Section \ref{s:uni-est}, we obtain uniform bounds on the sequence of Galerkin solutions, and use these in Section \ref{s:limit} in order to take the limit as $n \to \infty$. Finally, we show that the limit is the correct one in Section \ref{s:id-non}. Uniqueness is then proved in Section \ref{s:un}. In Section \ref{regularity} we obtain further temporal and spatial regularity estimates for these solutions. Finally, in Section \ref{limiting} we look at the case when $p=1$ and give the proof of Theorem~\ref{T4.1}.

\section{Proof of Theorem~\ref{T1}}\label{Thm1}
To prove the existence of a weak solution, we use a compactness argument based on a sequence of Galerkin approximations. However, since $\bG$ is not invertible in general, we introduce the following regularization:
$$
\bG_n(\bT):= \bG(\bT) + n^{-1} |\bT|^{p-2}\,\bT.
$$
Note that for all $n\in \mathbb{N}$, the regularized mapping still satisfies \eqref{A1}--\eqref{A3} (with $C_2$ replaced by $(C_2+1)$) and in addition the inequality \eqref{A1} is strict whenever $\bT\neq \bW$. Therefore, it directly follows from the theory of monotone operators that there exists a continuous inverse $\bG_n^{-1}:\Rsym \to \Rsym$.

%Assume for simplicity that $g$ is strictly monotone (if not, then we can add a term $\delta {T}^{p-2} T$ and let $\delta \to 0_{+})$. That is, we assume that
%\begin{equation}\label{str-mon}
%\big(g(T_{1}) - g(T_{2})\big) \cdot (T_{1} - T_{2}) > 0\quad \text{for}\quad T_{1} \neq T_{2}.
%\end{equation}
%Then, since $g$ is continuous, there exists a continuous inverse and hence we want
%\[T=g^{-1}(\alpha \epsilon + \beta \epsilon_{t}).\]

\subsection{Galerkin approximation}\label{galerkin}

Let $\{\bom_{j}\}_{j=1}^{\infty}$ be a basis of $W_{0}^{2d, 2}(\Omega;\R^d)$, which is orthonormal in $L^{2}(\Omega; \R^d)$.\footnote{Such a basis can be found by looking for eigenfunctions $\bom_j \in W_{0}^{2d, 2}(\Omega;\R^d)$ of the problem
$$
-\Delta^{2d} \bom_j = \lambda_j \bom_j\qquad \mbox{on $\Omega$.}$$
}  We denote by $P^n$ the projection of $W^{2d,2}_0(\Omega;\R^d)$ onto the linear hull of $\{\bom_{j}\}_{j=1}^{n}$, which is continuous.
We look for $\bu^{n}$ of the form
\[\bu^{n}(t, x) = \bu_0(t,x) + \sum_{i=1}^{n} C_{i}^{n}(t) \bom_{i}(x),\]
by solving, for all $j = 1, 2, \ldots, n$ and almost all $t\in (0,T)$, the following problem:
\begin{subequations}\label{Gn}
\begin{align}
\int_{\Omega} \partial^2_{tt} \bu^{n} \cdot \omega_{j} +\bG_n^{-1}\left(\alpha \beps(\bu^{n}) + \beta \partial_{t} \beps(\bu^{n})\right) \cdot \nabla \bom_{j} \dx &= \langle \bef, \bom_j\rangle, \label{G1} \\
 \bu^{n}(0) &= \bu_{0}(0),\label{G2}\\
 \partial_{t} \bu^{n}(0) &= \partial_t \bu_0(0). \label{G3}
\end{align}
\end{subequations}
We note that \eqref{G2} and \eqref{G3} are equivalent to $\vec{C}^n(0)=\mathbf{0}$ and $\partial_t \vec{C}^n(0)=\mathbf{0}$, respectively.
Since $\bG_n^{-1}$ is continuous and the basis functions $\{\bom_{j}\}_{j=1}^\infty $ are orthonormal in $L^2(\Omega;\mathbb{R}^d)$, the equation \eqref{G1} reduces to
\[\partial_{tt} C^{n}_{i}(t) = F_i(t,\vec{C}^{n}(t), \partial_{t} \vec{C}^{n}(t) ),\]
where $F_i$ are Carath\'{e}odory mappings. Hence, using  standard Carath\'{e}odory theory for a system of ordinary differential equations, we deduce that there exists a solution on some maximal  time interval $(0, T^{*})$. Furthermore, either we must have  $|\vec{C}^{n}(t)| + |\partial_{t}\vec{C}^{n}(t)| \to \infty$ as $t \to T^{*}_{-}$ or we can extend the solution to the whole interval $(0, T)$. We shall next show that the latter is true by establishing uniform bounds on the sequence of Galerkin approximations.

\subsection{Uniform bounds}\label{s:uni-est}

First, let us define
\[
\bT^n:=\bG_n^{-1}\left(\alpha \beps(\bu^{n}) + \beta \partial_{t} \beps(\bu^{n})\right),
\]
which is clearly equivalent to
\begin{equation}\label{Tn}
\alpha \beps(\bu^{n}) + \beta \partial_{t} \beps(\bu^{n})=\bG(\bT^n) + n^{-1} |\bT^n|^{p-2}\bT^n.
\end{equation}
Then, we multiply \eqref{G1} by $C^n_j$ and also by $\partial_t C^n_j$ and sum the resulting identities with respect to $j=1,\ldots, n$ to obtain
\begin{equation}\label{test1}
\begin{split}
 \int_{\Omega} \partial_{tt} \bu^n \cdot \partial_t (\bu^n - \bu_0)+ \bT^n \cdot \partial_t \beps(\bu^n-\bu_0)\dx  &= \langle \bef, \partial_t (\bu^n-\bu_0)\rangle, \\
 \int_{\Omega} \partial_{tt}\bu^n \cdot (\bu^n-\bu_0)+ \bT^n \cdot \beps(\bu^n-\bu_0)\dx &= \langle \bef, (\bu^n - \bu_0)\rangle.
\end{split}
\end{equation}
Next, it follows from \eqref{Tn} that
\[\bT^n \cdot \partial_t \beps(\bu^n) = \frac{1}{\beta} \left(\bG(\bT^n) \cdot \bT^n + n^{-1}|\bT^n|^p - \alpha \, \bT^n \cdot \beps(\bu^n)\right).\]
Also, we can write
\[ \int_{\Omega} \partial_{tt}(\bu^n-\bu_0) \cdot (\bu^n-\bu_0) \dx  =\ddt \int_{\Omega} \partial_{t}(\bu^n-\bu_0) \cdot (\bu^n-\bu_0)\dx - \int_{\Omega} | \partial_{t} (\bu^n-\bu_0)|^{2}\dx.\]
Using these two identities in \eqref{test1}, we obtain
\begin{equation}\label{test2}
\begin{split}
 &\frac{1}{2} \ddt\int_{\Omega}|\partial_{t} (\bu^n-\bu_0)|^2 \dx + \frac{1}{\beta} \int_{\Omega} \bG(\bT^n)\cdot \bT^n + n^{-1}|\bT^n|^p \dx \\
 &= \frac{\alpha}{\beta} \int_{\Omega}  \bT^n \cdot \beps(\bu^n)\dx + \langle \bef, \partial_t (\bu^n-\bu_0)\rangle +\int_{\Omega} \bT^n\cdot \partial_t \beps(\bu_0) - \partial_{tt} \bu_0 \cdot \partial_t(\bu^n-\bu_0)\dx,
 \end{split}
\end{equation}
and
\begin{equation}\label{test2b}
\begin{split}
&\ddt\frac{\alpha}{\beta} \int_{\Omega} \partial_{t}(\bu^n-\bu_0)\cdot (\bu^n-\bu_0)\dx + \frac{\alpha}{\beta} \int_{\Omega} \bT^n \cdot \beps(\bu^n)\dx \\
&= \frac{\alpha}{\beta} \int_{\Omega} |\partial_{t} (\bu^n-\bu_0)|^{2}-\partial_{tt}\bu_0 \cdot (\bu^n-\bu_0) +\bT^n \cdot \beps(\bu_0)\dx +\frac{\alpha}{\beta}\langle \bef, (\bu^n - \bu_0)\rangle.
\end{split}
\end{equation}
By summing these equalities we find that one term cancels and we deduce that
\begin{equation}\label{test3}
\begin{split}
 &\frac{1}{2} \ddt\int_{\Omega}|\partial_{t} (\bu^n-\bu_0)|^2 +\frac{2\alpha}{\beta}  \partial_{t}(\bu^n-\bu_0)\cdot (\bu^n-\bu_0)\dx+ \frac{1}{\beta} \int_{\Omega} \bG(\bT^n)\cdot \bT^n + n^{-1}|\bT^n|^p \dx \\
 &=\langle\bef, \partial_t (\bu^n-\bu_0)\rangle +\int_{\Omega} \bT^n\cdot \partial_t \beps(\bu_0) - \partial_{tt} \bu_0 \cdot \partial_t(\bu^n-\bu_0)\dx\\
 &\quad +\frac{\alpha}{\beta} \int_{\Omega} |\partial_{t} (\bu^n-\bu_0)|^{2}-\partial_{tt}\bu_0 \cdot (\bu^n-\bu_0) +\bT^n \cdot \beps(\bu_0)\dx +\langle \bef, (\bu^n - \bu_0)\rangle.
 \end{split}
\end{equation}
Next, we define on \( [0, T]\) the function
$$
Y^n:=\frac14\int_{\Omega}|\partial_{t}(\bu^n-\bu_0)|^{2} + |\bu^n-\bu_0|^{2}  + \left|\partial_{t} (\bu^n-\bu_0) + \frac{2 \alpha}{\beta} (\bu^n-\bu_0)\right|^2\dx.
$$
Using this, we can rewrite the first term on the left-hand side of \eqref{test3} as
$$
\frac{1}{2} \ddt\int_{\Omega}|\partial_{t} (\bu^n-\bu_0)|^2 +\frac{2\alpha}{\beta}  \partial_{t}(\bu^n-\bu_0)\cdot (\bu^n-\bu_0)\dx=Y^n - \left(\frac{\alpha^{2}}{\beta^{2}} + \frac{1}{4} \right)\ddt \int_{\Omega}  |\bu^n-\bu_0|^2 \dx.
$$
Consequently, using this identity in \eqref{test3}, applying \eqref{A2} to the second term on the left-hand side, and the H\"{o}lder inequality to the terms on the right-hand side together with the Poincar\'e and Korn inequalities, it follows that
\begin{equation}
\begin{split}
 &\ddt Y^n + \frac{C_1}{\beta} \int_{\Omega} |\bT^n|^p \dx \le C\left(1+ Y^n\right)+C(\|\beps(\bu_0)\|_{p'}+\|\partial_t \beps(\bu_0)\|_{p'})\|\bT^n\|_p\\
 &+ C(\|\beps(\bu^n)\|_{p'}+\|\partial_t \beps(\bu^n)\|_{p'}+\|\beps(\bu_0)\|_{p'}+\|\partial_t \beps(\bu_0)\|_{p'})(\|\bef\|_{(W^{1,p'}_0)^*} +\|\partial_{tt} \bu_0\|_{(W^{1,p'}_0)^*}),
\end{split}\label{tt1}
\end{equation}
where $C$ is a generic constant that is independent of \( n \).
To bound  the right-hand side, we use \eqref{Tn} to observe that
\[\partial_{t} \left(\mathrm{e}^{\frac{\alpha}{\beta} t} \beps(\bu^n)\right) = \frac{\mathrm{e}^{\frac{\alpha}{\beta} t}}{\beta} (\bG(\bT^n)+n^{-1}|\bT^n|^{p-2}\bT^n),\]
which, after integration with respect to time, gives
\[\beps(\bu^n(t))= \mathrm{e}^{-\frac{\alpha}{\beta} t} \beps(\bu_{0}(0)) + \mathrm{e}^{-\frac{\alpha}{\beta} t} \int_{0}^{t} \frac{\mathrm{e}^{\frac{\alpha}{\beta}\tau}}{\beta} (\bG(\bT^n(\tau) + n^{-1}|\bT^n(\tau)|^{p-2}\bT^n(\tau))\dtau.\]
Using  properties of the Bochner integral, it follows that
\begin{equation}
\begin{split}
\|\beps(\bu^n(t))\|_{p'}^{p'} &\leq C  \left(\int_{0}^{t} \| \bG(\bT^n) + n^{-1}|\bT^n|^{p-2}\bT^n\|_{p'}^{p'}\dtau  + \| \bu_{0}(0) \|_{1,p'}^{p'} \right)\\
&\leq C  \left(\int_{0}^{t} \|\bT^n\|_p^p\dtau  + \| \bu_{0}(0) \|_{1,p'}^{p'}+1 \right), \label{tt3}
\end{split}
\end{equation}
where for the second inequality we have used \eqref{A3}. Consequently, using  \eqref{tt3} and \eqref{Tn}, we have also the following bound on the time derivative:
\begin{equation}
\begin{split}
\|\partial_t \beps(\bu^n(t))\|_{p'}^{p'} &\leq C  \left(1+  \| \bu_{0}(0) \|_{1,p'}^{p'}+ \|\bT^n(t)\|_p^p +\int_{0}^{t} \|\bT^n\|_p^p\dtau \right). \label{tt4}
\end{split}
\end{equation}
Hence, using \eqref{tt3} and \eqref{tt4} for the terms appearing on the right-hand side of \eqref{tt1}, and applying Young's inequality to the resulting right-hand side, we see that
\begin{equation}
\begin{split}
 &\ddt \left(Y^n + \frac{C_1}{4\beta} \int_0^t \|\bT^n\|_p^p \dtau\right) + \frac{C_1}{4\beta}  \|\bT^n\|_p^p \le C\left(Y^n + \frac{C_1}{4\beta} \int_0^t \|\bT^n\|_p^p \dtau\right)\\
 &\qquad +C\sup_{t\in [0,T]}\|\bu_0(t)\|_{1,p'}^{p'} + C\left(\|\partial_t \beps(\bu_0)\|^{p'}_{p'}+\|\bef\|^p_{(W^{1,p'}_0)^*} +\|\partial_{tt} \bu_0\|^p_{(W^{1,p'}_0)^*}\right).
\end{split}\label{tfinal}
\end{equation}
Thus, using Gr\"{o}nwall's lemma and the assumptions on the data, we have that
\begin{equation}
\begin{split}
\sup_{t\in (0,T)} Y^n(t)  + \int_0^T\|\bT^n\|_p^p \dtau\le C(\bu_0,\bef)+ Y^n(0)=C(\bu_0,\bef).
\end{split}\label{tfinali}
\end{equation}
Finally, from the definition of $Y^n$,  the bounds \eqref{tt3}, \eqref{tt4}, and Korn's inequality, we deduce that
\begin{equation}
\sup_{t\in (0,T)} \left(\|\partial_t \bu^n\|_2^2+\|\bu^n\|_2^2 + \|\bu^n\|_{1,p'}^{p'}\right)+\int_0^T \|\bT^n\|_p^{p} + \|\partial_t \bu^n\|_{1,p'}^{p'} \dt \le C(\bu_0,\bef).\label{AE-n}
\end{equation}
It remains to provide a bound on $\partial_{tt}\bu^n$. We define the set $\mathcal{V}:=\{\bw \in W^{2d, 2}_{0}(\Omega;\R^d), \| \bw\|=1\}$.  Using the orthonormality of the basis and  the continuity of $P^n$,  we deduce from \eqref{G1} that
\begin{equation*}
\begin{split}
 \| \partial_{tt} \bu^n(t) \|_{(W^{2d, 2}_{0}(\Omega;\R^d))^{*}} &= \sup_{\bw \in \mathcal{V}} \int_{\Omega} \partial_{tt} \bu^{n}(t)\cdot  \bw \dx \\
& = \sup_{\bw \in \mathcal{V}} \int_{\Omega} \partial_{tt} \bu^{n}(t) \cdot  P^{n} \bw \dx \\
& = \sup_{\bw \in \mathcal{V}}\left(\langle \bef, \bw\rangle - \int_{\Omega} \bT^{n}(t) \cdot  \nabla (P^{n} \bw) \dx\right)\\
& \leq \sup_{\bw \in \mathcal{V}} (\|\bef(t)\|_{(W^{1,p'}_0(\Omega;\R^d))^*}+\|\bT^{n} (t)\|_{p}) \| P^{n} \bw \|_{1,p'} \\
& \leq C\sup_{\bw \in \mathcal{V}}(\|\bef(t)\|_{(W^{1,p'}_0(\Omega;\R^d))^*}+\|\bT^{n} (t)\|_{p})  \|P^n\bw\|_{2d,2} \\
&\leq C (\|\bef(t)\|_{(W^{1,p'}_0(\Omega;\R^d))^*}+\|\bT^{n} (t)\|_{p}),
\end{split}
\end{equation*}
where we have used the fact that $W^{2d,2}(\Omega;\R^d)$ is continuously embedded into $W^{1,p'}(\Omega;\R^d)$.
Therefore, it follows from \eqref{AE-n} that
\begin{equation}\label{AE-2-n}
\int_{0}^{T} \| \partial_{tt} \bu^{n} \|^{p}_{(W^{2d, 2}_{0}(\Omega;\R^d))^{*}} \dt\le C\int_0^T \|\bef\|^p_{(W^{1,p'}_0(\Omega;\R^d))^*}+\|\bT^{n} \|_{p}^p \dt \le C(\bu_0,\bef).
\end{equation}

\subsection{Limit $n\to \infty$}\label{s:limit}

 Using the bounds from Section \ref{s:uni-est} in conjunction with the reflexivity and separability of the underlying spaces,  we can find a subsequence, that we do not relabel, such that
\begin{equation}\label{C1}
\begin{aligned}
\bG(\bT^{n}) & \rightharpoonup \bar{\bG} &&\text{weakly in }L^{p'}(0, T; L^{p'}(\Omega;\Rsym)), \\
\bu^{n} & \overset{*}{\rightharpoonup} \bu &&\text{weakly$^*$ in } W^{1, \infty}(0, T; L^{2}(\Omega;\R^d)), \\
\bu^{n} & \rightharpoonup \bu &&\text{weakly in } W^{1, p'}(0, T; W^{1, p'}(\Omega;\R^d)), \\
\bT^{n} & \rightharpoonup \bT &&\text{weakly in } L^{p}(0, T; L^{p}(\Omega;\Rsym)), \\
\partial_{tt} \bu^{n} & \rightharpoonup \partial_{tt} \bu &&\text{weakly in } L^{p}(0, T; (W_{0}^{2d, 2}(\Omega;\R^d))^{*}).
\end{aligned}
\end{equation}
Hence, we see that $\bT$ fulfills \eqref{FST} and $\bu$ belongs to the first two spaces indicated in \eqref{FSu}. In addition, thanks to the fact that $W^{1,p'}(\Omega;\R^d)$ is compactly embedded into $L^2(\Omega;\R^d)$, using the Aubin--Lions lemma (for a further subsequence, not indicated,) we even have that
\begin{equation}\label{C2}
\begin{aligned}
\bu^{n} & \to  \bu &&\text{strongly in }\mathcal{C}([0,T];L^2(\Omega;\R^d), \\
\partial_{t} \bu^{n} & \to \partial_{t} \bu &&\text{strongly in }L^{2}(0, T; L^{2}(\Omega;\R^d))\cap \mathcal{C}([0,T]; (W^{2d,2}_0(\Omega; \R^d))^* ).
\end{aligned}
\end{equation}
Thus, it follows directly from the fact that $\bu^n(0)=\bu_0(0)$ and $\partial_t \bu^n(0)=\partial_t \bu_0(0)$ and the above convergence result (\ref{C2}) that
$$
\bu(0)=\bu_0 \quad \textrm{and} \quad \partial_t \bu(0)=\partial_t \bu_0(0).
$$

Next, we let $n\to \infty$ in \eqref{G1}. Let $\phi \in \mathcal{C}^{\infty}([0, T])$ be arbitrary. We multiply \eqref{G1} by $\phi$ and integrate the result  over $(0, T)$ to get
\[\int_{0}^{T} \langle \partial_{tt} \bu^{n}, \bom_{j} \phi \rangle \dt + \int_{0}^{T} \int_{\Omega} \bT^{n} \cdot \nabla(\phi \bom_{j})\dx \dt  = \int_0^T \langle \bef, \bom_j \rangle \dt ,\]
for every \( j\in \{ 1,\ldots, n \}\).
Thus, for a fixed $j$, we can let $n \to \infty$ and using the weak convergence result \eqref{C1} we deduce that
\[\int_{0}^{T} \langle \partial_{tt} \bu, \bom_{j} \phi \rangle \dt + \int_{0}^{T} \int_{\Omega} \bT \cdot \nabla(\phi \bom_{j})\dx \dt  = \int_0^T \langle \bef, \bom_j \rangle \dt.\]
Since $j$ and $\phi$ were arbitrary and recalling that $\{\bom_{j}\}_{j=1}^{\infty}$ forms a basis of $W^{2d,2}_0(\Omega; \R^d)$, it follows that
\begin{equation}\label{WFe}
\langle \partial_{tt}\bu, \bw \rangle + \int_{\Omega} \bT \cdot \nabla \bw \dx = \langle \bef, \bw \rangle \qquad \forall\, \bw \in W_{0}^{2d,2}(\Omega;\R^d),\quad \text{for a.e. }\,t \in (0, T).
\end{equation}
Consequently, thanks to the density of $W^{2d,2}_0(\Omega; \R^d)$ in $W^{1,p'}_0(\Omega;\R^d)$, we see that for almost all $t\in (0,T)$ we have $\partial_{tt}\bu \in (W^{1,p'}_0(\Omega;\R^d))^*$. Furthermore, we have
$$
\|\partial_{tt}\bu^n(t)\|_{(W^{1,p'}_0(\Omega;\R^d))^*}=\sup_{\bw\in W^{1,p'}_0(\Omega;\R^d); \, \|\bw\|=1}\left[ -\int_{\Omega}\bT^n(t)\cdot \nabla \bw\dx +\langle \bef (t), \bw \rangle\right].
$$
Thus, using \eqref{AE-n} and \eqref{C1}, it follows that
\begin{equation}
\int_0^T \|\partial_{tt}\bu^n\|_{(W^{1,p'}_0(\Omega;\R^d))^*}^p \dt \le C\int_0^T \|\bT^n\|_p^p + \|\bef\|_{(W^{1,p'}_0(\Omega;\R^d))^*}^p \dt\le C(\bu_0,\bef).\label{time-u}
\end{equation}
Hence, \eqref{WFe} can be strengthened so that \eqref{WF} holds. In addition, by standard parabolic interpolation and the fact that $\partial_t \bu_0 \in \mathcal{C}([0,T]; L^2(\Omega; \R^d))$, we see that  $\bu$ satisfies \eqref{FSu}.

Finally, letting $n\to \infty$ in \eqref{Tn} and using \eqref{C1}, we see that
\begin{equation}\label{T-constw}
\alpha \beps(\bu) + \beta \partial_{t} \beps(\bu)=\overline{\bG} \quad \textrm{a.e. in }Q.
\end{equation}
Hence, in order to show \eqref{T-const} and deduce the existence of a weak solution, it remains to show that $\overline{\bG}=\bG(\bT)$ a.e. in $Q$.

\subsection{Identification of the nonlinearity}\label{s:id-non}

In order to identify the nonlinearity, we will use monotone operator theory.  Let $\phi\in \mathcal{C}^1_0([0,T])$ be an arbitrary nonnegative function. We multiply both identities in \eqref{test1} by $\phi$ and integrate the result over $(0,T)$. With the help of integration by parts and the fact that $\bu^n(0)=\bu_0(0)$ and $\phi(T)=0$, we observe that
\begin{equation}\label{id1}
\begin{split}
 &\int_0^T\int_{\Omega} \bT^n \cdot \partial_t \beps(\bu^n) \phi\dx \dt=\int_0^T\int_{\Omega}\frac{|\partial_{t} (\bu^n-\bu_0)|^2 \phi'}{2}+ \bT^n\cdot \partial_t \beps(\bu_0)\phi\dx\dt \\
 &\qquad+\int_0^T \langle \bef, \partial_t (\bu^n-\bu_0)\rangle\phi - \langle\partial_{tt} \bu_0,  \partial_t(\bu^n-\bu_0)\rangle\phi \dt
 \end{split}
\end{equation}
and
\begin{equation}\label{id2}
\begin{split}
\int_0^T\int_{\Omega} \bT^n \cdot \beps(\bu^n)\phi\dx \dt&=\int_0^T\int_{\Omega} \partial_{t}(\bu^n-\bu_0)\cdot (\bu^n-\bu_0)\phi'\dx\dt  \\
&\quad+\int_0^T\int_{\Omega} |\partial_{t} (\bu^n-\bu_0)|^{2}\phi+\bT^n \cdot \beps(\bu_0)\phi\dx \dt \\
&\quad+\int_0^T\langle \bef, (\bu^n - \bu_0)\rangle\phi -\langle\partial_{tt}\bu_0, (\bu^n-\bu_0)\rangle  \phi \dt.
\end{split}
\end{equation}
Next, we use the weak convergence results \eqref{C1} and the strong convergence results \eqref{C2} to identify the limits on the right-hand sides of \eqref{id1} and \eqref{id2}. In particular, we obtain
\begin{equation}\label{id1l}
\begin{split}
 \lim_{n\to \infty}&\int_0^T\int_{\Omega} \bT^n \cdot \partial_t \beps(\bu^n) \phi\dx \dt=\int_0^T\int_{\Omega}\frac{|\partial_{t} (\bu-\bu_0)|^2 \phi'}{2}+\bT\cdot \partial_t \beps(\bu_0)\phi\dx\dt \\
 & \qquad +\int_0^T \langle \bef, \partial_t (\bu-\bu_0)\rangle\phi - \langle\partial_{tt} \bu_0,  \partial_t(\bu-\bu_0)\rangle\phi \dt
 \end{split}
\end{equation}
and
\begin{equation}\label{id2l}
\begin{split}
\lim_{n\to \infty}\int_0^T\int_{\Omega} \bT^n \cdot \beps(\bu^n)\phi\dx \dt&=\int_0^T\int_{\Omega} \partial_{t}(\bu-\bu_0)\cdot (\bu-\bu_0)\phi'\dx\dt  \\
&\quad+\int_0^T\int_{\Omega} |\partial_{t} (\bu-\bu_0)|^{2}\phi+\bT \cdot \beps(\bu_0)\phi\dx \dt \\
&\quad+\int_0^T\langle \bef, (\bu - \bu_0)\rangle\phi -\langle\partial_{tt}\bu_0, (\bu-\bu_0)\rangle  \phi \dt.
\end{split}
\end{equation}
Next, we use \eqref{WF} to evaluate the terms on the right-hand sides of \eqref{id1l}, \eqref{id2l}.
We note that, thanks to the regularity of $\bu$, both
$\bu-\bu_0$ and $\partial_{t} (\bu-\bu_0)$
are admissible test functions in \eqref{WF}. Using these two choices as the test function $\bw$, multiplying each of the resulting equalities by $\phi$ and integrating over $(0, T)$, we may apply integration by parts in order to obtain the following identities:
\begin{equation}\label{id1lf}
\begin{split}
 &\int_0^T\int_{\Omega} \bT \cdot \partial_t \beps(\bu) \phi\dx \dt=\int_0^T\int_{\Omega}\frac{|\partial_{t} (\bu-\bu_0)|^2 \phi'}{2}+\bT\cdot \partial_t \beps(\bu_0)\phi\dx\dt \\
 & \qquad +\int_0^T \langle \bef, \partial_t (\bu-\bu_0)\rangle\phi - \langle\partial_{tt} \bu_0,  \partial_t(\bu-\bu_0)\rangle\phi \dt
 \end{split}
\end{equation}
and
\begin{equation}\label{id2lf}
\begin{split}
\int_0^T\int_{\Omega} \bT \cdot \beps(\bu)\phi\dx \dt&=\int_0^T\int_{\Omega} \partial_{t}(\bu-\bu_0)\cdot (\bu-\bu_0)\phi'\dx\dt  \\
&\quad+\int_0^T\int_{\Omega} |\partial_{t} (\bu-\bu_0)|^{2}\phi+\bT \cdot \beps(\bu_0)\phi\dx \dt \\
&\quad+\int_0^T\langle \bef, (\bu - \bu_0)\rangle\phi -\langle\partial_{tt}\bu_0, (\bu-\bu_0)\rangle  \phi \dt.
\end{split}
\end{equation}
Comparing \eqref{id1l} with \eqref{id1lf} and \eqref{id2l} with \eqref{id2lf}, we see that
\begin{equation}
\label{jk}
\limsup_{n\to \infty} \int_{Q} \phi\bT^{n} \cdot (\alpha \beps(\bu^{n}) + \beta \partial_t \beps(\bu^{n}))\dx \dt  \leq \int_{Q} \phi\bT \cdot (\alpha \beps(\bu)  + \beta \partial_t \beps(\bu))\dx \dt.
\end{equation}
Therefore, using the nonnegativity of $\phi$, we observe that
\begin{equation}\label{MN}
\begin{split}
\limsup_{n\to \infty}\int_{Q} \phi\bG(\bT^{n})\cdot \bT^{n}\dx \dt   &\le\limsup_{n\to \infty}\int_{Q} \phi(\bG(\bT^{n})+n^{-1}|\bT^n|^{p-2}\bT^n)\cdot \bT^{n}\dx \dt \\
&\overset{\eqref{Tn}}=\limsup_{n\to \infty}\int_{Q} \phi\bT^n \cdot (\alpha \beps(\bu^{n}) + \beta \partial_t \beps(\bu^{n})) \dx \dt \\
&\overset{\eqref{jk}}\le \int_{Q} \phi\bT \cdot (\alpha \beps(\bu) + \beta \partial_t \beps(\bu)) \dx \dt\\
&\overset{\eqref{T-constw}}=\int_{Q} \phi\bT \cdot \overline{\bG} \dx \dt.
\end{split}
\end{equation}
The inequality (\ref{MN}) is the key to identifying the nonlinearity.
Let $\bW \in L^{p}(Q,\Rsym)$ be arbitrary. Using the  monotonicity assumption \eqref{A1}, the weak convergence results  \eqref{C1}, the bound  \eqref{MN} and the nonnegativity of $\phi$, we obtain
\[0 \leq \limsup_{n\to \infty} \int_{Q} \phi \left(\bG(\bT^{n}) - \bG(\bW)\right)\cdot\left(\bT^{n} - \bW\right)\dx \dt  \leq \int_{Q} \phi \left(\overline{\bG} - \bG(\bW)\right)\cdot (\bT - \bW)\dx \dt.\]
Setting $\bW = \bT - \kappa \bB$ for an arbitrary $\bB \in L^{p'}(Q;\Rsym)$ and $\kappa>0$, we divide through by $\kappa$ to deduce that
$$
0\le  \int_{Q} \phi \left(\overline{\bG} - \bG(\bT - \kappa \bB)\right)\cdot \bB\dx \dt.
$$
Hence, since $\bG$ is continuous, we may let $\kappa \to 0_+$ to deduce that
$$
0\le  \int_{Q} \phi \left(\overline{\bG} - \bG(\bT)\right)\cdot \bB\dx \dt.
$$
As $\bB$ and $\phi$ are arbitrary, we conclude that
\[\overline{\bG} = \bG(\bT) \quad \text{a.e. in }Q.\]
Thus we have proved the existence of a weak solution.

\subsection{Uniqueness of solution}\label{s:un}
To complete the proof of Theorem \ref{T1} it remains to show uniqueness of the weak solution.
To this end, let $(\bu_{1}, \bT_{1})$ and $(\bu_{2}, \bT_{2})$ be two weak solutions of (\ref{model}) emanating from the same data. We denote $\bu:=\bu_1-\bu_2$.  Then, using \eqref{WF}, we see that
\[\langle \partial_{tt} \bu, \bw \rangle + \int_{\Omega} (\bT_{1} - \bT_{2})\cdot \beps(\bw)\dx = 0 \quad \forall\, \bw \in W_{0}^{1, p'}(\Omega;\R^d) \textrm{ and a.e. }t\in (0,T).\]
Since $\bu$ and $\partial_t \bu$ belong to $W^{1,p'}_0(\Omega;\R^d)$ for almost all $t\in (0,T)$, we can set $\bw=\bu$ and $\bw= \partial_{t} \bu$ in the above to deduce that, for almost all $t$, the following holds:
\begin{equation*}
\begin{split}
&\frac{1}{2} \ddt \| \partial_{t} \bu\|_{2}^{2} + \int_{\Omega} (\bT_{1} - \bT_{2}) \cdot \partial_t \beps(\bu) \dx = 0, \\
& \ddt \int_{\Omega} \partial_{t} \bu \cdot  \bu \dx  + \int_{\Omega} (\bT_{1} - \bT_{2}) \cdot \beps(\bu)\dx  = \int_{\Omega}|\partial_{t} \bu |^2 \dx.
\end{split}
\end{equation*}
Therefore,
\[\ddt\left(\int_{\Omega} \frac{\beta}{2} | \partial_{t} \bu |^{2} + \alpha \partial_{t} \bu \cdot \bu \dx \right) + \int_{\Omega} (\bT_{1} - \bT_{2}) \cdot \left(\beta \partial_t\beps(\bu) + \alpha \beps(\bu)\right)\dx = \int_{\Omega} \alpha|\partial_{t} \bu|^2 \dx.\]
Using the same procedure as in the a~priori estimates and also the constitutive relation \eqref{T-const}, we obtain
\begin{equation*}
\begin{split}
& \frac{1}{4} \ddt \int_{\Omega} \beta |\partial_{t} \bu|^{2} + \beta |\bu|^{2} + \beta \left|\partial_{t} \bu + \frac{2 \alpha}{\beta} \bu \right|^{2}\dx + \int_{\Omega} (\bG(\bT_1)-\bG(\bT_2))\cdot (\bT_{1} - \bT_{2}) \,\mathrm{d}x\\
& =\int_{\Omega} \alpha|\partial_{t} \bu|^2 + \left(\beta+\frac{\alpha^2}{\beta}\right)|\bu|^2\dx
\\&
\le C(\alpha,\beta) \int_{\Omega} \beta |\partial_{t} \bu|^{2} + \beta |\bu|^{2} + \beta \left|\partial_{t} \bu + \frac{2 \alpha}{\beta} \bu \right|^{2}\dx.
\end{split}
\end{equation*}
The second term on the left-hand side is nonnegative thanks to \eqref{A1} so we may apply Gr\"{o}nwall's inequality.   Since $\bu(0)=\partial_t \bu(0)=\mathbf{0}$,  we deduce that $\bu = \mathbf{0}$ a.e. in $Q$. In addition, by monotonicity, we also obtain that  $\big(\bG(\bT_{1}) - \bG(\bT_{2})\big)\cdot (\bT_{1} - \bT_{2}) = 0$ a.e. in~$Q$. This proves that $\bu_{1} = \bu_{2}$ a.e. in $Q$, and if $\bG$ is strictly monotone then also $\bT_{1} = \bT_{2}$.

\section{Regularity estimates}\label{regularity}
\def\mA{\mathcal{A}}

In this section we prove the higher regularity estimates for the solution from Theorem \ref{T1}. We note that this is an essential part in the proof of the existence of a solution for the limiting strain model, i.e., the case~$p=1$, as it
involves passing to the limit $p \rightarrow 1_+$.

As the focus turns to the limiting strain model, in this part we will now assume that there exists a strictly convex $\mathcal{C}^2$-function $F:\Rsym\to \R^d$ such that, for all $\bT\in \Rsym$,
\begin{equation}\label{basic-e}
\frac{\partial F(\bT)}{\partial \bT} = \bG(\bT).
\end{equation}
In this case, $\bG$ is strictly monotone. To simplify the subsequent notation, for an arbitrary $\bT\in \Rsym$, we denote
$$
\mathcal{A}(\bT):=\frac{\partial^2 F(\bT)}{\partial \bT \partial \bT} = \frac{\partial \bG(\bT)}{\partial \bT}, \qquad \mA^{ij}_{kl}(\bT):=\frac{\partial \bG_{ij}(\bT)}{\partial \bT_{kl}}.
$$
We also define a new scalar product on $\Rsym$ by
\begin{equation}\label{scalar}
(\bV,\bW)_{\mA}:= \mA(\bT)\bV \cdot \bW = \sum_{i,j,k,l=1}^d \frac{\partial \bG_{ij}(\bT)}{\partial \bT_{kl}} \bV_{ij} \bW_{kl}.
\end{equation}
The fact that (\ref{scalar}) does indeed define a scalar product follows from the fact that~$\bG$ has a  potential~$F$. In particular, we know that for all $\bT\in \Rsym$ there holds $\frac{\partial \bG_{ij}(\bT)}{\partial \bT_{kl}}  =\frac{\partial \bG_{kl}(\bT)}{\partial \bT_{ij}}$, i.e., symmetry, and also   $\mA$ is positive definite as a result of the convexity assumption.

In what follows, we will split the regularity estimates. First, we focus on time regularity and then we consider regularity with respect to the spatial variable. Here, we provide only a formal proof. Nevertheless, the time regularity proof is in fact fully rigorous since it can be deduced already at the level of Galerkin approximations. The spatial regularity proof is only formal, but can be justified by using a standard difference quotient technique. We emphasise  that we do not impose any coercivity and growth assumptions on $\mA$ here because, in the case $p=1$, we lose such information.

We note that when $p\in (1,\infty)$ one can usually assume that
\begin{equation}
\begin{aligned}
|(\bV,\bW)_{\mA}|\le C_3(1+|\bT|)^{p-2}\, |\bV|\, |\bW|,
\qquad (\bW,\bW)_{\mA} \ge C_4(1+|\bT|)^{p-2}\,
|\bW|^2.
\end{aligned}\label{simplas}
\end{equation}
Under  assumption \eqref{simplas}, the regularity estimates can be deduced in an easier way. However, they are not included here as the more challenging case of $p=1$ is our primary interest.
Also, it is worth  observing that our prototype models  \eqref{A4p} do not satisfy \eqref{simplas}$_2$ and in general, the assumption \eqref{simplas}$_2$ cannot be satisfied when $p=1$. %\textcolor{red}{\bf (VP: What are our prototype models? I don't think it is explicitly stated anywhere so maybe remove the last sentence?)}

Defining the convex conjugate \(F^*\) of \( F\) as in Section~\ref{physics}, we recall that, from the definition of~\( \bG\), we have that
\begin{equation}
F(\bT) + F^*(\bG(\bT)) = \bG(\bT) \cdot \bT.\label{basic-e2}
\end{equation}

\subsection{Time regularity}
Here, we improve the bound on the time derivative. This bound will be used for the limiting strain model in order to pass to the limit in the term $\partial_{tt}\bu$ in the weak formulation. We formulate the following lemma locally in time in order to keep the initial data as general as possible.
\begin{lemma}\label{time-r}
Let $p\in (1,\infty)$ and suppose that \eqref{basic-e} holds with $\bG$ fulfilling \eqref{A1}--\eqref{A3}. Assume that $\bef\in L^2(0,T; L^2(\Omega;\R^d))$ and $\bu_0 \in W^{2,p'}_{loc}(0,T; W^{1,p'}(\Omega;\R^d))$. Then for any weak solution to~\eqref{model} and for every~$\delta>0$, the following bound holds:
\begin{equation}\label{WFfinallocal}
\begin{split}
&\sup_{t\in (\delta,T)}\int_{\Omega} F^*(\bG(\bT))\dx+\int_{\delta}^T\|\partial_{tt}\bu\|_2^2\dt\\
& \le  C(\alpha,\beta) \left(\int_{\frac{\delta}{2}}^T\int_{\Omega} |\bef|_2^2 + |\partial_t \bu|_2^2 + |\partial_{tt} \bu_0|_2^2 + |\partial_t \bu_0|_2^2+|\bT \cdot \partial_t(\beta\partial_t \beps(\bu_0)+\alpha \beps(\bu_0))|\dx \dt\right)\\
&\quad  +\frac{C(\alpha,\beta)}{\delta}\int_0^{\delta}\int_{\Omega}F^*(\alpha \beps(\bu(\tau)) + \beta \partial_t \beps(\bu(\tau)))+|\partial_t \bu(\tau)|^2\dx \dtau.
\end{split}
\end{equation}
If additionally $\bu_0 \in W^{2,p'}(0,T; W^{1,p'}(\Omega;\R^d))$, then we have the following global-in-time bound:
\begin{equation}\label{WFfinal}
\begin{split}
&\sup_{t\in (0,T)}\int_{\Omega} F^*(\bG(\bT))\dx+\int_0^T\|\partial_{tt}\bu\|_2^2\dt\\
& \le  C(\alpha,\beta) \left(\int_Q |\bef|_2^2 + |\partial_t \bu|_2^2 + |\partial_{tt} \bu_0|_2^2 + |\partial_t \bu_0|_2^2+|\bT \cdot \partial_t(\beta\partial_t \beps(\bu_0)+\alpha \beps(\bu_0))|\dx \dt\right)\\
&\quad  +C(\alpha,\beta)\int_{\Omega}F^*(\alpha \beps(\bu_0(0)) + \beta \partial_t \beps(\bu_0(0))) +  |\partial_t \bu_0(0)|^2\dx.
\end{split}
\end{equation}
\end{lemma}

\begin{proof}
Recalling  that  $\bef\in L^2(0,T;L^2(\Omega,\R^d))$, we set $\bw:=\partial_{tt}(\bu-\bu_0)$ in \eqref{WF} to observe  that, for almost all $t\in (0,T)$,
\begin{equation*}
\int_{\Omega} |\partial_{tt}\bu|^2+ \bT \cdot \partial_{tt}\beps(\bu)\dx = \int_{\Omega} \bef\cdot \partial_{tt}(\bu-\bu_0) + \partial_{tt}\bu\cdot \partial_{tt}\bu_0+\bT \cdot \partial_{tt}\beps(\bu_0)\dx.
\end{equation*}
This identity can be rewritten as
\begin{equation}
\begin{split}\label{wtf}
&\beta\|\partial_{tt}\bu\|_2^2+ \int_{\Omega}\bT \cdot (\beta \partial_{tt}\beps(\bu)+ \alpha\partial_t\beps(\bu))\dx \\
&\quad = \beta\int_{\Omega} \bef\cdot \partial_{tt}(\bu-\bu_0) + \partial_{tt}\bu\cdot \partial_{tt}\bu_0+\bT \cdot \partial_{tt}\beps(\bu_0)+\frac{\alpha}{\beta}\bT \cdot \partial_t\beps(\bu)\dx.
\end{split}
\end{equation}
First, we evaluate the last term on the right-hand side. Setting $\bw:=\partial_t(\bu-\bu_0)$ in \eqref{WF}, we see that
\begin{equation*}
\int_{\Omega}\bT \cdot \partial_{t}\beps(\bu)\dx =-\frac12 \ddt \|\partial_t \bu\|_2^2 + \int_{\Omega} \bef\cdot \partial_{t}(\bu-\bu_0) + \partial_{tt}\bu\cdot \partial_{t}\bu_0+\bT \cdot \partial_{t}\beps(\bu_0)\dx.
\end{equation*}
From the second term on the left-hand side of (\ref{wtf}), using \eqref{cons-law}, we see that
$$
\begin{aligned}
\int_{\Omega}\bT \cdot (\beta \partial_{tt}\beps(\bu)+ \alpha\partial_t\beps(\bu))\dx&=\int_{\Omega} \bT\cdot \partial_t \bG(\bT)\,\mathrm{d}x=\int_{\Omega}\partial_t(\bT\cdot \bG(\bT)) - \partial_t \bT \cdot \bG(\bT)\dx \\
&=\int_{\Omega}\partial_t(\bT\cdot \bG(\bT) - F(\bT))\dx=\ddt \int_{\Omega} F^*(\bG(\bT))\dx.
\end{aligned}
$$
Thus, using these two identities in \eqref{wtf} and applying Young's inequality, we obtain the following:
\begin{equation}\label{WFt2}
\begin{split}
&\ddt \left(\int_{\Omega} F^*(\bG(\bT))+\frac{\alpha |\partial_t \bu|^2}{2\beta}\dx  \right)+\frac{\beta}{2}\|\partial_{tt}\bu\|_2^2\\
&\quad \le  C(\alpha,\beta) (\|\bef\|_2^2 + \|\partial_t \bu\|_2^2 + \|\partial_{tt} \bu_0\|_2^2 + \|\partial_t \bu_0\|_2^2) + \frac{1}{\beta}\int_{\Omega} \bT \cdot \partial_t(\beta\partial_t \beps(\bu_0)+\alpha \beps(\bu_0)).
\end{split}
\end{equation}
Integrating (\ref{WFt2}) over \( (0, T)  \) and using the fact that
$$
 F^*(\bG(\bT(0)))=F^*(\alpha \beps(\bu_0) + \beta \partial_t \beps(\bu_0)),
$$
we deduce \eqref{WFfinal}. Similarly, integrating \eqref{WFt2} over $(\tau,t)$ where $\delta/2\le \tau\le \delta\le t\le T$ are arbitrary, we deduce that
\begin{equation}\label{WFt221}
\begin{split}
&\sup_{t\in (\delta,T)}\left(\int_{\Omega} F^*(\bG(\bT))+\frac{\alpha |\partial_t \bu|^2}{2\beta}\dx  \right)+\int_{\delta}^T\|\partial_{tt}\bu\|_2^2\,\mathrm{d}t\\
&\quad \le  C(\alpha,\beta) {\int_{\frac{\delta}{2}}^T \int_{\Omega}} |\bef|^2 + |\partial_t \bu|^2 + |\partial_{tt} \bu_0|^2 + |\partial_t \bu_0|^2 +|\bT \cdot \partial_t(\beta\partial_t \beps(\bu_0)+\alpha \beps(\bu_0))|\dx \dt \\
&\qquad + C(\alpha, \beta)\int_{\Omega}F^*(\alpha \beps(\bu(\tau)) + \beta \partial_t \beps(\bu(\tau))) + |\partial_t \bu(\tau)|^2\dx.
\end{split}
\end{equation}
Integrating with respect to $\tau \in (\delta/2,\delta)$ and dividing by $\delta$,  we directly obtain  \eqref{WFfinallocal}.
\end{proof}

\subsection{Spatial regularity}
Here, we will improve the spatial regularity of a weak solution. In particular,
we prove a weighted bound on $\nabla \bT$, which is a key tool for obtaining the existence of a weak solution  for the limiting strain model, i.e., in the case $p=1$.
\begin{lemma}
\label{reg-s}
Let all of the assumptions of Lemma~\ref{time-r} be satisfied. In addition, assume that $\partial_t \bu_0(0)\in W^{1,2}(\Omega;\R^d)$ and
$$
\int_0^T\int_{\Omega} |\mA(\bT)||\bT|^2 +|\mA(\bT)||\bef|^2\dx\dt< \infty.
$$
Then, for an arbitrary open set $\Omega' \subset \overline{\Omega'} \subset \Omega$ and any $\delta>0$, we have the following bound:
\begin{equation}\label{sp-r-l}
\begin{split}
&\sup_{t\in (\delta,T)} \|\partial_t \nabla \bu\|_{L^2(\Omega')} + \sum_{k=1}^d\int_{\delta}^T \int_{\Omega'}(\partial_k \bT, \partial_k \bT)_{\mA(\bT)} \dx \dt \\
&\quad \le  C(\Omega',\delta)\int_0^T\int_{\Omega} |\bT| |\bG(\bT)| + |\mA(\bT)||\bT|^2 + |\bef|^2 + |\nabla \bu|^2 +|\partial_t \nabla \bu|^2 +|\mA(\bT)||\bef|^2\dx\dt.
\end{split}
\end{equation}
If, additionally,  $\bu_0 \in \mathcal{C}^1([0,T]; W^{1,2}(\Omega;\R^d))$, then we also have
\begin{equation}\label{sp-r}
\begin{split}
&\sup_{t\in (0,T)} \|\partial_t \nabla \bu\|_{L^2(\Omega')} + \sum_{k=1}^d\int_0^T \int_{\Omega'}(\partial_k \bT, \partial_k \bT)_{\mA(\bT)} \dx \dt \\
&\quad \le  C(\Omega')\int_0^T\int_{\Omega} |\bT| |\bG(\bT)| + |\mA(\bT)||\bT|^2 + |\bef|^2 + |\nabla \bu|^2 +|\partial_t \nabla \bu|^2 +|\mA(\bT)||\bef|^2\dx\dt\\
&\qquad + C\|\partial_t \nabla \bu_0(0)\|_2^2.
\end{split}
\end{equation}

\end{lemma}

\begin{proof}
%In this section we improve the time regularity.
Fix an arbitrary nonnegative smooth compactly supported $\varphi\in \mathcal{C}^{\infty}_0(\Omega)$. Then, we can choose $\bw:= -\diver (\varphi^2 \nabla \partial_t \bu)$ in \eqref{WF}  and integrate by parts to deduce the following identity:
\begin{equation}
\label{startr}
\begin{split}
\frac{\beta}{2}\ddt \int_{\Omega}|\partial_t \nabla \bu \varphi|^2\dx  + \int_{\Omega}\sum_{i,j,k=1}^d \partial_{k}\bT_{ij} \partial_j(\varphi^2 \beta \partial_t \partial_k \bu_i) \dx = -\beta\int_{\Omega} \bef\cdot \diver (\varphi^2 \nabla \partial_t \bu)\dx.
\end{split}
\end{equation}
Similarly, setting $\bw:= -\diver (\varphi^2 \nabla  \bu)$ in \eqref{WF} leads to
\begin{equation}
\label{startr2}
\begin{split}
&\alpha \ddt \int_{\Omega} \partial_{t}\nabla \bu \cdot \nabla \bu \varphi^2\dx    + \int_{\Omega}\sum_{i,j,k=1}^d \partial_{k}\bT_{ij} \partial_j(\varphi^2 \alpha \partial_k \bu_i) \dx \\
&\qquad = -\alpha\int_{\Omega} \bef\cdot \diver (\varphi^2 \nabla \bu)\dx+\alpha\int_{\Omega} |\partial_{t}\nabla\bu \varphi|^2\dx.
\end{split}
\end{equation}
Summing these two identities, we deduce that
\begin{equation}
\label{startr3}
\begin{split}
\frac{\beta}{4}\ddt \int_{\Omega}|\partial_t \nabla \bu \varphi|^2 +\left|\partial_t \nabla \bu \varphi+\frac{2\alpha}{\beta}\nabla \bu\varphi\right|^2\dx    + \int_{\Omega}\sum_{i,j,k=1}^d \partial_{k}\bT_{ij} \partial_j(\varphi^2 (\alpha \partial_k \bu_i+\beta \partial_t \partial_k \bu_i)) \dx  \\
= -\int_{\Omega} \bef\cdot \diver (\varphi^2 (\beta\nabla \partial_t \bu+\alpha\nabla \bu))\dx+\frac{2\alpha^2}{\beta} \int_{\Omega}\partial_{t}\nabla \bu\cdot \nabla \bu \varphi^2\dx+\alpha\int_{\Omega} |\partial_{t}\nabla\bu \varphi|^2\dx.
\end{split}
\end{equation}
Now we  show that the second integral on the left-hand side is the key source of information. We use \eqref{cons-law}, integration by parts and the symmetry of $\bT$ in order  to observe that
\begin{equation}\label{sttt}
\begin{aligned}
&\int_{\Omega}\sum_{i,j,k=1}^d \partial_{k}\bT_{ij} \partial_j(\varphi^2 (\alpha \partial_k \bu_i+\beta \partial_t \partial_k \bu_i)) \dx\\
&=\sum_{i,j,k=1}^d \int_{\Omega}\partial_{k}\bT_{ij} (\varphi^2 (\alpha \partial_k \partial_j\bu_i+\beta \partial_t \partial_k \partial_j\bu_i)) +2\partial_{k}\bT_{ij} \varphi\partial_j\varphi (\alpha \partial_k \bu_i+\beta \partial_t \partial_k \bu_i) \dx\\
&=\sum_{i,j,k=1}^d \int_{\Omega}\partial_{k}\bT_{ij} \varphi^2 \partial_k(\alpha \beps_{ij}(\bu)+\beta \partial_t \beps_{ij}(\bu)) +4\partial_{k}\bT_{ij} \varphi\partial_j\varphi (\alpha \beps_{ik}(\bu)+\beta \partial_t \beps_{ik}(\bu)) \dx\\
&\quad -2\sum_{i,j,k=1}^d \int_{\Omega}\partial_{k}\bT_{ij} \varphi\partial_j\varphi (\alpha \partial_i \bu_k+\beta \partial_t \partial_i \bu_k) \dx\\
&=\sum_{i,j,k=1}^d \int_{\Omega}\partial_{k}\bT_{ij} \varphi^2 \partial_k\bG_{ij}(\bT) -4\bT_{ij} \partial_{k}(\varphi\partial_j\varphi) \bG_{ik}(\bT)-4\bT_{ij} \varphi\partial_j\varphi \partial_{k}\bG_{ik}(\bT) \dx\\
&\quad +\sum_{i,j,k=1}^d \int_{\Omega}\bT_{ij} \partial_{kj}(\varphi^2 )\partial_i(\alpha  \bu_k+\beta \partial_t  \bu_k) \dx+2\sum_{i,j,k=1}^d \int_{\Omega}\bT_{ij} \varphi\partial_j\varphi \partial_i(\alpha  \partial_{k}\bu_k+\beta \partial_t  \partial_{k}\bu_k)\dx\\
&=\int_{\Omega}\sum_{k=1}^d (\partial_k \bT \varphi, \partial_k \bT \varphi)_{\mA(\bT)} -4\sum_{i,j,k=1}^d \bT_{ij} \partial_{k}(\varphi\partial_j\varphi) \bG_{ik}(\bT)-4\sum_{i,j,k=1}^d \bT_{ij} \varphi\partial_j\varphi \partial_{k}\bG_{ik}(\bT) \dx\\
&\quad -\sum_{i,j,k=1}^d \int_{\Omega}\partial_j\bT_{ij} \partial_{k}(\varphi^2 )\partial_i(\alpha  \bu_k+\beta \partial_t  \bu_k) \dx-\sum_{i,j,k=1}^d \int_{\Omega}\bT_{ij} \partial_{k}(\varphi^2 )\partial_{ij}(\alpha  \bu_k+\beta \partial_t  \bu_k) \dx\\
&\quad +2\sum_{i,j,k=1}^d \int_{\Omega}\bT_{ij} \varphi\partial_j\varphi \partial_i \bG_{kk}(\bT)\dx
\\&=: \sum_{m=1}^6 I_m.
\end{aligned}
\end{equation}
We need to determine what bounds can be deduced from \eqref{sttt}. In particular, we show that the terms $I_2, \ldots, I_6$ can be bounded in terms of $I_1$ and the data. The simplest bound  is for $I_2$. In particular,  it directly follows that
$$
|I_2| \le C(\varphi) \int_{\Omega} |\bT|\, |\bG(\bT)|\dx.
$$
Letting $\delta_{nk}$ denote the Kronecker delta, in order to bound \( I_3\) we first rewrite it in the following way:
$$
\begin{aligned}
\sum_{i,j,k=1}^d \bT_{ij} \varphi\partial_j\varphi \partial_{k}\bG_{ik}(\bT)&=\sum_{i,j,k,l,m,n=1}^d \delta_{nk}\bT_{ij} \varphi\partial_j\varphi \mA^{ik}_{lm}(\bT)\partial_{n}\bT_{lm}\\
&=\sum_{j,n=1}^d \left(\sum_{i,k,l,m=1}^d \mA^{ik}_{lm}(\bT)\partial_{n}\bT_{lm}  \delta_{nk}\bT_{ij} \varphi\partial_j\varphi\right).
\end{aligned}
$$
Using the Cauchy--Schwarz inequality and the fact that $\mA$ generates a scalar product, by applying Young's inequality we find that
$$
\begin{aligned}
|I_3| &\le C \int_{\Omega}\left|\sum_{j,n=1}^d \left(\sum_{i,k,l,m=1}^d \mA^{ik}_{lm}(\bT)\partial_{n}\bT_{lm}  \delta_{nk}\bT_{ij} \varphi\partial_j\varphi\right) \right|\dx \\
&\le C \int_{\Omega}\left|\sum_{j,n=1}^d \left(\sum_{i,k,l,m=1}^d \mA^{ik}_{lm}(\bT)\partial_{n}\bT_{lm} \varphi \partial_{n}\bT_{ik} \varphi  \right)^{\frac12}\left(\sum_{i,k,l,m=1}^d \mA^{ik}_{lm}(\bT) \delta_{nm}\bT_{lj} \partial_j \varphi \delta_{nk}\bT_{ij} \partial_j\varphi\right)^{\frac12} \right|\dx\\
&\le \frac{I_1}{8} +C(\varphi) \int_{\Omega}|\mA(\bT)| |\bT|^2\dx.
\end{aligned}
$$
The term $I_6$ can be bounded in a very similar way. In particular, we have
$$
|I_6|\le \frac{I_1}{8} +C(\varphi) \int_{\Omega}|\mA(\bT)| |\bT|^2\dx.
$$
For $I_4$, we use the equation \eqref{linear-moment} and  Young's inequality to obtain
$$
\begin{aligned}
|I_4|&=\left|\sum_{i,k=1}^d \int_{\Omega}(\bef_i -\partial_{tt} \bu_i) \partial_{k}(\varphi^2 )\partial_i(\alpha  \bu_k+\beta \partial_t  \bu_k) \dx\right|\\
&\le C(\varphi) \int_{\Omega} |\bef|^2 + |\partial_{tt}\bu|^2 + |\partial_t \nabla \bu \varphi|^2 + |\nabla \bu \varphi|^2 \dx.
\end{aligned}
$$
Finally, to evaluate $I_5$, we first recall the following identity
\begin{equation}
\begin{aligned}
&\partial_{ij}(\alpha  \bu_k+\beta \partial_t  \bu_k)\\
&= \partial_i(\alpha\beps_{jk}(\bu)+\beta \partial_t\beps_{jk}(\bu)) +\partial_j(\alpha\beps_{ik}(\bu)+\beta \partial_t\beps_{ik}(\bu))-\partial_k(\alpha\beps_{ij}(\bu)+\beta \partial_t\beps_{ij}(\bu)).
\end{aligned}\label{ID12}
\end{equation}
Then, we can rewrite $I_5$ with the help of \eqref{cons-law} to find that
$$
I_5=-\sum_{i,j,k=1}^d \int_{\Omega}\bT_{ij} \partial_{k}(\varphi^2 )\left(\partial_i \bG_{jk}(\bT) +\partial_j \bG_{ik}(\bT)-\partial_k\bG_{ij}(\bT)\right) \dx.
$$
Hence, we see that we are in the same situation as with the term $I_3$ and we can deduce that
$$
|I_5|\le \frac{I_1}{8} +C(\varphi) \int_{\Omega}|\mA(\bT)| |\bT|^2\dx.
$$
Thus we have suitable bounds on the left-hand side of (\ref{startr3}).
Next we rewrite the first term on the right-hand side of \eqref{startr3} in the following way:
\begin{equation*}
%\label{ops}
\begin{split}
&\int_{\Omega} \bef\cdot \diver (\varphi^2 (\alpha \nabla \bu +\beta \partial_t \nabla \bu))\dx\\
&=\sum_{i,j=1}^d \int_{\Omega} \bef_i (\partial_j (\varphi^2) (\alpha \partial_j \bu_i +\beta \partial_t \partial_j\bu_i)+\varphi^2 (\alpha \partial_{jj}\bu_i +\beta \partial_t \partial_{jj} \bu_i))\dx\\
&=\sum_{i,j=1}^d \int_{\Omega} \bef_i (\partial_j (\varphi^2) (\alpha \partial_j \bu_i +\beta \partial_t \partial_j\bu_i)+\varphi^2(2\partial_j \bG_{ij}(\bT)-\partial_i \bG_{jj}(\bT))\dx.
\end{split}
\end{equation*}
Hence, using Young's inequality in the first term and a procedure similar to the one used for $I_3$ in the second, we get
\begin{equation}
\label{ops}
\begin{split}
&\left|\int_{\Omega} \bef\cdot \diver (\varphi^2 (\alpha \nabla \bu +\beta \partial_t \nabla \bu))\dx\right|\\
&\quad \le \frac{I_1}{8}+C(\varphi)\int_{\Omega} |\bef|^2 + |\nabla \bu|^2 +|\partial_t \nabla \bu|^2 +|\mA(\bT)||\bef|^2\dx.
\end{split}
\end{equation}

Substituting the above bounds into \eqref{startr3} and using a similar procedure to the one used in the proof of Lemma~\ref{time-r}, we deduce \eqref{sp-r} and \eqref{sp-r-l}.
\end{proof}

\section{Limiting strain - Proof of Theorem~\ref{T4.1}}\label{limiting}

%{\tt \color{red}This is the real start of the section}

%{\tt \color{red} Add to (A4) that $|\phi''(s)|\le C(1+s)^{-1}$}

As in the proof of Theorem \ref{T1}, in order to prove Theorem \ref{T4.1} we first introduce an approximate problem. However, we are able to make  use of the knowledge obtained from Theorem \ref{T1}. Indeed, we define a function on \( \Rsym\) by
\begin{equation}\label{Gndf}
\bG^n(\bT):= \bG(\bT) + n^{-1} \bT.
\end{equation}
Since $\bG$ satisfies \eqref{A1}--\eqref{A3} with $p=1$, it is evident that $\bG^n$ satisfies \eqref{A1}--\eqref{A3} with $p=2$. Therefore, thanks to Theorem~\ref{T1}, we know that there exists a couple $(\bu^n,\bT^n)$, fulfilling\footnote{We assume a slightly different restriction on $\bu_0$ than in~Theorem~\ref{T1}. However, the proof of Theorem~\ref{T1} can be easily adapted  to this case.}
\begin{align}
\bu^n &\in \mathcal{C}^{1}([0, T]; L^{2}(\Omega;\R^d))\cap W^{1, 2}(0, T; W^{1,2}(\Omega;\R^d))\cap  W^{2,2}(0, T; (W_{0}^{1,2 }(\Omega;\R^d))^{*}), \label{FSun}\\
\bT^n &\in L^{2}(0, T; L^{2}(\Omega;\Rsym))\label{FSTn}
\end{align}
and satisfying
\begin{equation}\label{WFn}
\langle \partial_{tt}\bu^n, \bw \rangle + \int_{\Omega} \bT^n \cdot \nabla \bw \dx = \int_{\Omega} \bef \cdot \bw\dx \qquad \forall\, \bw \in W_{0}^{1,2}(\Omega;\R^d)\quad \text{for a.e. }\,t \in (0, T),
\end{equation}
and
\begin{equation}\label{T-constn}
\alpha \beps(\bu^n) + \beta \partial_{t} \beps(\bu^n)=\bG^n(\bT^n)= \bG(\bT^n) + n^{-1}\bT^n \quad \textrm{a.e. in }Q.
\end{equation}
We note that we can replace the duality pairing by the integral over $\Omega$ in the term containing $\bef$ thanks to the assumed regularity of $\bef$.
Moreover, we know that\footnote{In case that $\Omega$ is not a Lipschitz domain, the identity below is not understood in the sense of traces but in the sense that $\bu-\bu_0 \in W^{1,1}_0(\Omega;\R^d)$ for almost all $t\in (0,T)$, where $W^{1,1}_0(\Omega;\R^d)$ defined as the closure of $C^\infty_0(\Omega;\R^d)$ in the norm of $W^{1,1}(\Omega;\R^d)$.}
$$
\bu^n=\bu_0 \quad \textrm{ on } \Gamma \cup (\{0\}\times \Omega), \qquad \partial_t\bu^n=\partial_t\bu_0 \quad \textrm{ on } \{0\}\times \Omega.
$$
We want to consider the limit as  $n\to \infty$ in order  to prove the existence of a solution to the limiting strain problem  in the sense of Theorem~\ref{T4.1}.

\subsection{A~priori $n$-independent bounds}
We start with bounds that are independent of the order of approximation. For this purpose, we use and mimic many steps from preceding sections. We start with the first uniform bound. Setting $\bw:=\beta \partial_t(\bu^n-\bu_0)+\alpha(\bu-\bu_0)$ in \eqref{WFn},  after exactly the same algebraic manipulations as those used for  \eqref{test3} we deduce that
\begin{equation}\label{test3n}
\begin{split}
 &\frac{1}{4} \ddt\int_{\Omega}\beta|\partial_{t} (\bu^n-\bu_0)|^2 +\beta\left|\partial_{t} (\bu^n-\bu_0)+\frac{2\alpha}{\beta}(\bu^n-\bu_0)\right|^2\dx+  \int_{\Omega} \bG^n(\bT^n)\cdot \bT^n  \dx \\
 &=\int_{\Omega} \bT^n\cdot (\alpha \beps(\bu_0)+\beta\partial_t \beps(\bu_0)) \dx  +\alpha \int_{\Omega} |\partial_{t} (\bu^n-\bu_0)|^{2}\,\mathrm{d}x\\
 &\quad +\int_{\Omega}(\bef-\partial_{tt} \bu_0) \cdot (\alpha(\bu^n-\bu_0)+ \beta\partial_t (\bu^n-\bu_0))\dx+ \frac{2\alpha^2}{\beta}\int_{\Omega}\partial_t (\bu^n-\bu_0) \cdot (\bu^n-\bu_0)\dx.
 \end{split}
\end{equation}
In order to obtain the required a~priori estimate, we need to use the safety strain condition. In particular, it follows from \eqref{compt} that there exists a $\delta>0$ such that
\begin{equation}
|\alpha \beps(\bu_0)+\beta\partial_t \beps(\bu_0)|\le L-2\delta \qquad \textrm{a.e. in }Q, \label{compt-A}
\end{equation}
where $L$ is as defined in \eqref{dfL2}. In addition, if we define \( F(\bT) := \phi(|\bT|) \), it  follows from the convexity of $\phi$  that, for any $\tilde{\delta}>0$, there exists a $C_{\tilde{\delta}}$ such that, for all $\bT\in \R^{d\times d}_{sym}$,
\begin{equation}
F(\bT)\ge (L-\tilde{\delta})|\bT| -C_{\tilde{\delta}}.\label{lowerbo}
\end{equation}
We choose \( \tilde{\delta} = \delta\) as in (\ref{compt-A}) and let \( C_\delta\) be the corresponding constant from (\ref{lowerbo}).
Since $C_{\delta}$ depends in principle on $\bu_0$ and $F$, and \( \delta\) is now given, we do not trace the dependence of $C$ on $\delta$ in what follows.
Consequently, for the second term on the left-hand side of \eqref{test3n}, we can use \eqref{basic-e2} and \eqref{T-constn} to deduce that
$$
\bG^n(\bT^n)\cdot \bT^n = n^{-1}|\bT^n|^2 + F(\bT^n) + F^*(\bG(\bT^n))\ge (L-\delta)|\bT^n|+ n^{-1}|\bT^n|^2 -C(\delta).
$$
Furthermore, the first term on the right-hand side of \eqref{test3n} can be  bounded by using \eqref{compt-A} in the following way:
$$
\int_{\Omega} \bT^n\cdot (\alpha \beps(\bu_0)+\beta\partial_t \beps(\bu_0)) \dx\le (L-2\delta)\|\bT^n\|_1.
$$
Therefore, it follows from \eqref{test3n}, the above bounds and  H\"{o}lder's inequality that
\begin{equation}\label{test3n2}
\begin{split}
 &\frac{1}{4} \ddt\int_{\Omega}\beta|\partial_{t} (\bu^n-\bu_0)|^2 +\beta\left|\partial_{t} (\bu^n-\bu_0)+\frac{2\alpha}{\beta}(\bu^n-\bu_0)\right|^2\dx+  \delta \|\bT^n\|_1+ n^{-1}\|\bT^n\|_2^2\\
 &\le C \left( \int_{\Omega}\beta|\partial_{t} (\bu^n-\bu_0)|^2 +\beta\left|\partial_{t} (\bu^n-\bu_0)+\frac{2\alpha}{\beta}(\bu^n-\bu_0)\right|^2\dx+\|\bef\|_2^2 + \|\partial_{tt}\bu_0\|_2^2 +1\right).
 \end{split}
\end{equation}
The application of Gr\"{o}nwall's lemma then leads to
\begin{equation}\label{aen1}
\sup_{t\in (0,T)} \left(\|\partial_t \bu(t)\|_2^2 + \|\bu(t)\|_2^2\right) + \int_0^T \|\bT^n\|_1 + n^{-1}\|\bT^n\|_2^2 \dt \le C(\bef,\bu_0),
\end{equation}
where we have used the assumption \eqref{data-as2} on the data. It also follows from \eqref{T-const} and the above bound  that
$$
\int_{Q}|\alpha \beps(\bu)+\beta\partial_t \beps(\bu)|^2 \dx \dt \le \int_{Q} (L+n^{-1}|\bT^n|)^2 \dx \dt \le C(\bef,\bu_0).
$$
Consequently, since $\beps(\bu(0))\in L^{\infty}(\Omega,\R^{d\times d})$ and $\bG^n(\bT^n)\cdot \bT^n$ is nonnegative, arguing similarly as in the bound  \eqref{tt3}, we deduce with the help of Korn's inequality and \eqref{test3n} that
\begin{equation}
\int_Q |\bG^n(\bT^n)\cdot \bT^n|\dx\dt +\int_0^T \|\partial_t\bu^n\|_{1,2}^2 + \|\bu^n\|_{1,2}^2\dt\le C(\bef,\bu_0).\label{aen2}
\end{equation}

\subsection{Regularity via $n$-independent bounds}
The bounds \eqref{aen1}, \eqref{aen2} are  not sufficient to pass to the limit $n\to \infty$, since we only have a~priori control on $\bT^n$ in a nonreflexive space $L^1(Q;\R^{d \times d})$. In particular, at best we have that  the weak star limit of $\bT^n$ is  a measure. Therefore, the pointwise relation  \eqref{T-const2} is neither meaningful nor likely to be valid in this case.  Instead, we improve our information by using the regularity technique introduced in Section~\ref{regularity}. Namely, we use Lemma~\ref{time-r} and Lemma~\ref{reg-s}. First, we define an approximatinon $F_n$ of the potential $F$ by
$$
F_n(\bT):= F(\bT) + \frac{|\bT|^2}{2n}.
$$
Then we  have that
$$
\frac{\partial F^n(\bT)}{\partial \bT}=\bG_n(\bT) = \bG(\bT) + n^{-1}\bT.
$$
We may now apply the results from Section~\ref{regularity} with $p=2$, replacing $(\bu, F, \bG)$ with the triple  $(\bu^n, F_n, \bG_n)$.  We note that using the definition of $\bG_n$ we may define~\( \mathcal{A}_n \) in an analogous way to~\( \mathcal{A}\). In particular, we write
$$
\begin{aligned}
(\mA_n(\bT^n))_{ijkl}&:=\frac{\partial}{\partial\bT^n_{kl}}\left(\frac{\phi'(|\bT^n|)}{|\bT^n|}\bT^n_{ij} + n^{-1} \bT^n_{ij}\right)\\
&=\delta_{ik}\delta_{jl} \left(n^{-1}+\frac{\phi'(|\bT^n|)}{|\bT^n|}\right)+ \left(\frac{\phi''(|\bT^n|)|\bT^n| - \phi'(|\bT^n|)}{|\bT^n|}\right)\frac{\bT^n_{ij}\bT^n_{kl}}{|\bT^n|^2}.
\end{aligned}
$$
Consequently, using the fact that $\phi'(0)=0$ and  $\phi''(s)\le C(1+s)^{-1}$, we see that
\begin{equation}\label{Abbound}
|\mA_n(\bT^n)| \le Cn^{-1} +\frac{C}{1+|\bT^n|}.
\end{equation}

With this in mind, let us first discuss regularity with respect to time. We see that all assumptions of Lemma~\ref{time-r} are satisfied. Therefore we have,  for every~$\delta>0$, that the following inequality holds:
\begin{equation}\label{WFfinallocaln}
\begin{split}
&\sup_{t\in (\delta,T)}\int_{\Omega} F_n^*(\bG_n(\bT^n))\dx+\int_{\delta}^T\|\partial_{tt}\bu^n\|_2^2\dt\\
& \le  C(\alpha,\beta) \left(\int_{\frac{\delta}{2}}^T\int_{\Omega} |\bef|_2^2 + |\partial_t \bu^n|_2^2 + |\partial_{tt} \bu_0|_2^2 + |\partial_t \bu_0|_2^2+|\bT^n \cdot \partial_t(\beta\partial_t \beps(\bu_0)+\alpha \beps(\bu_0))|\dx \dt\right)\\
&\quad  +\frac{C(\alpha,\beta)}{\delta}\int_0^{\delta}\int_{\Omega}F_n^*(\alpha \beps(\bu^n(\tau)) + \beta \partial_t \beps(\bu^n(\tau)))+|\partial_t \bu^n(\tau)|^2\dx \dtau.
\end{split}
\end{equation}
We focus on the bound  of the right-hand side. For the second integral on the right-hand side, it follows from the properties of the convex conjugate function and the uniform bounds \eqref{aen1}, \eqref{aen2} that
$$
\begin{aligned}
&\int_0^{\delta}\int_{\Omega}F_n^*(\alpha \beps(\bu^n) + \beta \partial_t \beps(\bu^n))+|\partial_t \bu^n|^2\dx \dtau
 =\int_0^{\delta}\int_{\Omega}F_n^*(\bG_n(\bT^n))+|\partial_t \bu^n|^2\dx \dtau
\\&\quad
\le \int_{Q}\bG_n(\bT^n)\cdot \bT^n+|\partial_t \bu^n|^2\dx \dt \le C(\bu_0,\bef).
\end{aligned}
$$
For the first term on the right-hand side of \eqref{WFfinallocaln}, we use H\"{o}lder's inequality, the assumptions on the data \eqref{data-as2}, \eqref{compt}, \eqref{compt2} and the uniform bound  \eqref{aen1} in order to deduce that
$$
\begin{aligned}
&\int_{\frac{\delta}{2}}^T\int_{\Omega} |\bef|_2^2 + |\partial_t \bu^n|_2^2 + |\partial_{tt} \bu_0|_2^2 + |\partial_t \bu_0|_2^2+|\bT^n \cdot \partial_t(\beta\partial_t \beps(\bu_0)+\alpha \beps(\bu_0))|\dx \dt\\
&\quad \le C(\bu_0,\bef) + \|| \partial_{tt}\beps(\bu_0)|+|\partial_t\beps(\bu_0)|\|_{L^{\infty}((\frac{\delta}{2},T)\times \Omega)}\int_{0}^T\int_{\Omega} |\bT^n|\dx \dt
\\&\quad
\le C(\bu_0, \bef).
\end{aligned}
$$
It follows from the above bounds and  \eqref{WFfinallocaln} that, for every $\delta>0$, we have
\begin{equation}\label{WFfinallocalnn}
\begin{split}
&\sup_{t\in (\delta,T)}\int_{\Omega} F_n^*(\bG_n(\bT^n))\dx+\int_{\delta}^T\|\partial_{tt}\bu^n\|_2^2\dt\le C(\bef,\bu_0).
\end{split}
\end{equation}

Similarly, in case that \eqref{compt2} holds even for $\delta=0$, we can use \eqref{WFfinal}. By a very similar computation to the one above we deduce that
\begin{equation}
\begin{split}
&\sup_{t\in (0,T)}\int_{\Omega} F_n^*(\bG_n(\bT^n))\dx+\int_0^T\|\partial_{tt}\bu^n\|_2^2\dt\\
&\quad \le  C(\bef,\bu_0) + C\int_{\Omega}F_n^*(\alpha \beps(\bu_0(0)) + \beta \partial_t \beps(\bu_0(0)))\dx \\
&\quad \le  C(\bef,\bu_0) + C\int_{\Omega}F^*(\alpha \beps(\bu_0(0)) + \beta \partial_t \beps(\bu_0(0)))\dx \\
&\quad \le C(\bef,\bu_0),
\end{split}\label{WFfinalnnn}
\end{equation}
using the fact that $F_n^* \le F^*$ and  assumptions \eqref{compt}, \eqref{compt2} with $\delta=0$.

Next, we consider the spatial regularity estimates. For an arbitrary open set $\Omega' \subset \overline{\Omega'} \subset \Omega$ and for any $\delta>0$, it follows from \eqref{sp-r-l} that
\begin{equation}\label{sp-r-ln}
\begin{split}
&\sup_{t\in (\delta,T)} \|\partial_t \nabla \bu^n\|_{L^2(\Omega')} + \sum_{k=1}^d\int_{\delta}^T \int_{\Omega'}(\partial_k \bT^n, \partial_k \bT^n)_{\mA_n(\bT^n)} \dx \dt \\
&\quad \le  C(\Omega',\delta)\int_Q |\bT^n| |\bG_n(\bT^n)| + |\mA_n(\bT^n)||\bT^n|^2 + |\bef|^2 + |\nabla \bu^n|^2 +|\partial_t \nabla \bu^n|^2 +|\mA_n(\bT^n)||\bef|^2\dx\dt.
\end{split}
\end{equation}
Since $|\bT^n| |\bG_n(\bT^n)|=|\bT^n\cdot \bG_n(\bT^n)|$, we can use \eqref{aen1}, \eqref{aen2} to deduce that
$$
\int_Q |\bT^n| |\bG_n(\bT^n)| + |\bef|^2 + |\nabla \bu^n|^2 +|\partial_t \nabla \bu^n|^2\dx\dt\le C(\bu_0,\bef).
$$
Thus, it only remains to bound the terms involving $\mA_n$ on the right-hand side of \eqref{sp-r-ln}.
To this end, we note that
$$
\int_Q |\mA_n(\bT^n)||\bT^n|^2 +|\mA_n(\bT^n)||\bef|^2\dx\dt \le C\int_{Q}n^{-1}|\bT^n|^2 + |\bT^n| + |\bef|^2\le C(\bu_0,\bef),
$$
where the last inequality follows from \eqref{aen1} and the assumptions on $\bef$. Using these inequalities for the terms appearing on the right-hand side of \eqref{sp-r-ln}, we immediately deduce that
\begin{equation}\label{sp-r-lnn}
\begin{split}
&\sup_{t\in (\delta,T)} \|\partial_t \nabla \bu^n\|_{L^2(\Omega')} + \sum_{k=1}^d\int_{\delta}^T \int_{\Omega'}(\partial_k \bT^n, \partial_k \bT^n)_{\mA_n(\bT^n)} \dx \dt \le C(\bu_0,\bef,\Omega').
\end{split}
\end{equation}
Similarly, if $\bu_0 \in \mathcal{C}^1([0,T]; W^{1,2}(\Omega;\R^d))$ we can use \eqref{sp-r} and perform similar computations to find that
\begin{equation}\label{sp-rn}
\begin{split}
&\sup_{t\in (0,T)} \|\partial_t \nabla \bu^n\|_{L^2(\Omega')} + \sum_{k=1}^d\int_0^T \int_{\Omega'}(\partial_k \bT^n, \partial_k \bT^n)_{\mA_n(\bT^n)} \dx \dt \le  C(\Omega',\bu_0,\bef).
\end{split}
\end{equation}
Next, we focus on the bounds on the second derivatives of $\partial_t \bu^n$ and $\bu^n$. It follows from \eqref{T-constn} and the Cauchy--Schwarz inequality that
$$
\begin{aligned}
&|\partial_k (\alpha \beps(\bu^n) + \beta\partial_t \beps(\bu^n))|^2=(\partial_k (\alpha \beps(\bu^n) + \beta\partial_t \beps(\bu^n)))\cdot \partial_k \bG_n(\bT^n) \\
&\quad = (\partial_k (\alpha \beps(\bu^n) + \beta\partial_t \beps(\bu^n)), \partial_k \bT^n)_{\mA_n(\bT^n)}\\
&\quad \le (\partial_k (\alpha \beps(\bu^n) + \beta\partial_t \beps(\bu^n)), \partial_k (\alpha \beps(\bu^n) + \beta\partial_t \beps(\bu^n)))_{\mA_n(\bT^n)}^{\frac12} (\partial_k \bT^n, \partial_k \bT^n)_{\mA_n(\bT^n)}^{\frac12}\\
&\quad \le C|\partial_k (\alpha \beps(\bu^n) + \beta\partial_t \beps(\bu^n))|(\partial_k \bT^n, \partial_k \bT^n)_{\mA_n(\bT^n)}^{\frac12}.
\end{aligned}
$$
Therefore,
$$
\begin{aligned}
&|\partial_k (\alpha \beps(\bu^n) + \beta\partial_t \beps(\bu^n))|^2\le C(\partial_k \bT^n, \partial_k \bT^n)_{\mA_n(\bT^n)}.
\end{aligned}
$$
Using this and \eqref{sp-r-lnn}, simple algebraic manipulations imply that
\begin{equation}\label{sp-r-lnn1}
\begin{split}
\int_{\delta}^T \int_{\Omega'} |\nabla(\alpha \beps(\bu^n) + \beta\partial_t \beps(\bu^n))|^2\dx \dt \le C(\bu_0,\bef,\Omega').
\end{split}
\end{equation}

\subsection{Convergence results as $n\to \infty$ based on uniform bounds}
From the uniform bounds \eqref{aen1}, \eqref{aen2}, we see that we can find a subsequence, that we do not relabel, such that
\begin{align}
\bu^n &\rightharpoonup \bu &&\textrm{weakly in }W^{1,2}(0,T; W^{1,2}(\Omega;\R^d)),\label{Kn1}\\
\bu^n &\rightharpoonup^* \bu &&\textrm{weakly$^*$ in } W^{1,\infty}(0,T; L^2(\Omega;\R^d)),\label{Kn2}\\
n^{-1}\bT^n &\to \b0 &&\textrm{strongly in }L^2(0,T; L^2(\Omega; \R^{d\times d})).\label{Kn3}
\end{align}
In addition, using the regularity estimates \eqref{WFfinallocalnn}, \eqref{sp-r-lnn1}, as well as the Aubin--Lions lemma, we deduce that
\begin{align}
\bu^n &\rightharpoonup \bu &&\textrm{weakly in }W^{2,2}_{loc}(0,T; L^{2}(\Omega;\R^d)),\label{Kn4}\\
\bu^n &\rightharpoonup \bu &&\textrm{weakly in } W^{1,2}_{loc}(0,T; W^{2,2}_{loc}(\Omega;\R^d)),\label{Kn5}\\
\bu^n &\to \bu &&\textrm{strongly in }W^{1,2}_{loc}(0,T; W^{1,2}_{loc}(\Omega; \R^{d})).\label{Kn6}
\end{align}
Next, we focus on the limiting passage in \eqref{T-constn}. Since the mapping $\bG$ is bounded, we know that
\begin{align}
\bG(\bT^n) &\rightharpoonup^* \overline{\bG} \qquad \textrm{weakly$^*$ in }L^{\infty}(Q;\R^{d\times d}). \label{Kn7}
\end{align}
Our goal is to identify $\overline{\bG}$. We first note that from \eqref{T-constn}, \eqref{Kn2} and \eqref{Kn3} we must have
\begin{equation}
\overline{\bG}=\alpha \beps(\bu) + \beta \partial_t \beps(\bu) \qquad \textrm{a.e. in }Q. \label{aennn}
\end{equation}
Next, we want to show that there exists a $\tilde{\bT}$ such that $\overline{\bG}=\bG(\tilde{\bT})$. To do so, we appeal to Chacon's biting lemma to deduce from \eqref{aen1} that there exists a $\bT\in L^1(Q; \R^{d\times d})$ and a nondecreasing sequence of sets $Q_1\subset Q_2\subset \cdots$, with $|Q\setminus Q_i|\to 0$ as $i\to \infty$, such that  for each $i\in \mathbb{N}$ there holds
\begin{align}
\bT^n &\rightharpoonup \bT \qquad\textrm{weakly in }L^{1}(Q_i;\R^{d\times d}).\label{Kn8}
\end{align}
However, thanks to \eqref{Kn6}, \eqref{aennn} and Egoroff's theorem, we also know that for every $\varepsilon>0$ and every $i\in \mathbb{N}$ there exists a $Q_{i,\varepsilon}\subset Q_i$, with $|Q_i\setminus Q_{i,\varepsilon}|\le \varepsilon$, such that
$$
\alpha \beps(\bu^n) + \beta \partial_t\beps(\bu^n)\to \overline{\bG}\qquad \textrm{strongly in } L^{\infty}(Q_{i,\varepsilon};\mathbb{R}^{d\times d}).
$$
Therefore, using the monotonicity of $\bG$ and the above convergence result, we deduce that, for an arbitrary $\bW\in L^1(Q;\R^{d\times d})$, that
$$
\begin{aligned}
0&\le \lim_{n\to \infty}\int_{Q_{i,\varepsilon}}(\bG(\bT^n)-\bG(\bW))\cdot (\bT^n-\bW)\dx \dt\\
&=   \int_{Q_{i,\varepsilon}}\bG(\bW)\cdot (\bW-\bT)-\overline{\bG}\cdot \bW\dx \dt+ \lim_{n\to \infty}\int_{Q_{i,\varepsilon}}\bG(\bT^n)\cdot \bT^n\dx \dt\\
&\le    \int_{Q_{i,\varepsilon}}\bG(\bW)\cdot (\bW-\bT)-\overline{\bG}\cdot \bW\dx \dt+ \lim_{n\to \infty}\int_{Q_{i,\varepsilon}}\bG_n(\bT^n)\cdot \bT^n\dx \dt\\
&=    \int_{Q_{i,\varepsilon}}\bG(\bW)\cdot (\bW-\bT)-\overline{\bG}\cdot \bW\dx \dt+ \lim_{n\to \infty}\int_{Q_{i,\varepsilon}}(\alpha \beps(\bu^n) + \beta \partial_t\beps(\bu^n))\cdot \bT^n\dx \dt\\
&=\int_{Q_{i,\varepsilon}}(\overline{\bG}-\bG(\bW))\cdot (\bT-\bW)\dx \dt.
\end{aligned}
$$
Since $\bG$ is a monotone mapping and $\bW$ is arbitrary, we can use Minty's method to deduce that
$$
\overline{\bG}=\bG(\bT) \qquad \textrm{ a.e. in }Q_{i,\varepsilon}.
$$
Recalling that $\varepsilon>0$ and $i\in \mathbb{N}$ are arbitrary, we see that \eqref{T-const2} follows from \eqref{aennn} and the above identity. Additionally, setting $\bW:=\bT$ in the above and using the fact that $\overline{\bG}=\bG(\bT)$, we see that
$$
\begin{aligned}
\lim_{n\to \infty}\int_{Q_{i,\varepsilon}}|(\bG(\bT^n)-\bG(\bT))\cdot (\bT^n-\bT)|\dx \dt= \lim_{n\to \infty}\int_{Q_{i,\varepsilon}}(\bG(\bT^n)-\bG(\bT))\cdot (\bT^n-\bT)\dx \dt=0.
\end{aligned}
$$
Consequently, we must have that
$$
\bT^n \to \bT \qquad \textrm{ a.e. in } Q_{i,\varepsilon},
$$
as a result of the strict monotonicity of $\bG$. However, as before, since $\varepsilon>0$ and $i\in \mathbb{N}$ are arbitrary, we deduce that
\begin{align}
\bT^n \to \bT \qquad \textrm{ a.e. in } Q. \label{Kn9}
\end{align}
Using \eqref{aen1}, \eqref{Kn9} and Fatou's lemma, it follows that
\begin{equation}
\int_{Q}|\bT|\dx \dt \le C(\bu_0,\bef).\label{aest1nl}
\end{equation}

Next, we focus on the boundary and initial conditions for $\bu$. It is evident from the convergence result \eqref{Kn1}, combined with the fact that $\bu^n=\bu_0$ on $\Gamma$ and  $\bu^n(0)=\bu_0$ on $\Omega$, that we must have $\bu=\bu_0$ on $\Gamma$ as well. Furthermore, it follows that
$$
\|\bu(t)-\bu_0(0)\|_{1,2} \to 0 \qquad \textrm{ as } t\to 0_+.
$$
Concerning the attainment of the initial condition for $\partial_t\bu(0)$ we need to proceed slightly differently since we have  control on $\partial_{tt}\bu$  locally in $(0,T)$.  We integrate \eqref{test3n} over a time interval $(0,t)$, where $0<t<T$, and since we know that for each $n$ the initial datum is attained we deduce that
\begin{equation}\label{test3naa}
\begin{split}
 &\frac{1}{4} \int_{\Omega}\beta|\partial_{t} (\bu^n-\bu_0)(t)|^2 +\beta\left|\partial_{t} (\bu^n-\bu_0)(t)+\frac{2\alpha}{\beta}(\bu^n-\bu_0)(t)\right|^2\dx \\
 &=\int_0^t\int_{\Omega} \bT^n\cdot ((\alpha \beps(\bu_0)+\beta\partial_t \beps(\bu_0))-\bG_n(\bT^n))  +\alpha  |\partial_{t} (\bu^n-\bu_0)|^{2}\dx \dtau\\
 & \quad+\int_0^t\int_{\Omega}(\bef-\partial_{tt} \bu_0) \cdot (\alpha(\bu^n-\bu_0)+ \beta\partial_t (\bu^n-\bu_0))+ \frac{2\alpha^2}{\beta}\partial_t (\bu^n-\bu_0) \cdot (\bu^n-\bu_0)\dx\dtau.
 \end{split}
\end{equation}
Our goal is to let $n\to \infty$. Since $t>0$, we can use the ``local" convergence result \eqref{Kn4} to let $n\to \infty$ in the left-hand side of \eqref{test3naa}. To bound  also the right-hand side, we first use the safety strain condition \eqref{compt}, which implies that there exists a $\bT_0\in L^1(Q;\R^{d\times d})$ such that
$$
\alpha \beps(\bu_0)+\beta\partial_t \beps(\bu_0)=\bG(\bT_0)\qquad \textrm{ a.e. in }Q.
$$
Thus, using the monotonicity of $\bG$, we see that
$$
\begin{aligned}
\bT^n\cdot ((\alpha \beps(\bu_0)+\beta\partial_t \beps(\bu_0))-\bG_n(\bT^n))&\le \bT^n\cdot (\bG(\bT_0)-\bG(\bT^n))\le \bT_0\cdot (\bG(\bT_0)-\bG(\bT^n)).
\end{aligned}
$$
Using the convergence results \eqref{Kn1}--\eqref{aennn} applied to all terms in \eqref{test3naa} with the above inequality  yields the following:
\begin{equation}\label{test3naab}
\begin{split}
 &\frac{1}{4} \int_{\Omega}\beta|\partial_{t} (\bu-\bu_0)(t)|^2 +\beta\left|\partial_{t} (\bu-\bu_0)(t)+\frac{2\alpha}{\beta}(\bu-\bu_0)(t)\right|^2\dx \\
 &\le\int_0^t\int_{\Omega} \bT_0\cdot ((\alpha \beps(\bu_0)+\beta\partial_t \beps(\bu_0))-\bG(\bT))  +\alpha  |\partial_{t} (\bu-\bu_0)|^{2}\dx\dtau\\
 & \quad +\int_0^t\int_{\Omega}(\bef-\partial_{tt} \bu_0) \cdot (\alpha(\bu-\bu_0)+ \beta\partial_t (\bu-\bu_0))+ \frac{2\alpha^2}{\beta}\partial_t (\bu-\bu_0) \cdot (\bu-\bu_0)\dx\dtau\\
 &\le C\int_0^t \|\bT_0\|_1 +\|\bef\|_2 +\|\partial_{tt}\bu_0\|_2 +1 \dtau.
 \end{split}
\end{equation}
Letting  $t\to 0_+$, we see that
$$
\lim_{t\to 0_+}(\|\bu(t)-\bu_0(0)\|_2^2 +\|\partial_t\bu(t)-\partial_t\bu_0(0)\|^2_2)=0.
$$
In addition, it also follows from \eqref{Kn4} that $\bu\in \mathcal{C}^{1}_{loc}(0,T; L^2(\Omega;\R^d))$, which combined with the above result gives that  $\bu\in \mathcal{C}^{1}([0,T]; L^2(\Omega;\R^d))$.

\subsection{Validity of the equation in the limit}
To summarize the results so far, we have found a couple $(\bu,\bT)$ that satisfies \eqref{FSu}--\eqref{FST2} and \eqref{T-const2}, \eqref{bcint2}. It remains to show \eqref{WF2}. To do so, we use the method developed in \cite{BeBuMaSu17}. Let $g$ be a smooth nonnegative nonincreasing  function satisfying
$$
g(s)=\left\{\begin{aligned}
&1, &&\textrm{for }s\in [0,1],\\
&0, &&\textrm{for }s>2.
\end{aligned}
\right.
$$
For each \( k \in \mathbb{N}\),  let us define
$$
g_k(s):=g(s/k).
$$
It is clear that $g_k\nearrow 1$. Next let $\bv\in \mathcal{C}^{\infty}_0(Q; \R^d)$ be arbitrary but fixed. Thanks to \eqref{Kn4} and \eqref{Kn9}, all terms in \eqref{WF2} are well-defined for almost all $t\in (0,T)$ and we just need to check that the equality holds.

Using the properties of $g_k$, we have
\begin{equation}\label{last1}
\begin{split}
I&:=\int_Q \partial_{tt}\bu \cdot \bv+ \bT \cdot \nabla \bv - \bef \cdot \bv \dx\dt \\
&= \lim_{k\to \infty}\int_Q \partial_{tt}\bu \cdot \bv g_k(|\bT|)+ \bT \cdot \nabla \bv g_k(|\bT|) - \bef \cdot \bv g_k(|\bT|) \dx\dt.
\end{split}
\end{equation}
Using \eqref{Kn4}, \eqref{Kn8}, the fact that $\bT^n\in L^2_{loc}(0,T; W^{1,2}_{loc}(\Omega; \R^{d\times d}))$, which follows from \eqref{sp-r-lnn}, and the fact that $g_k(|\bT^n|)$ is supported only in the set where $|\bT^n|\le k$, we can rewrite the right-hand side of \eqref{last1} in the following way:
  \begin{equation}\label{last2}
\begin{split}
I&= \lim_{k\to \infty}\lim_{n\to \infty}\int_Q \partial_{tt}\bu^n \cdot \bv g_k(|\bT^n|)+ \bT^n \cdot \nabla \bv g_k(|\bT^n|) - \bef \cdot \bv g_k(|\bT^n|) \dx\dt\\
&= \lim_{k\to \infty}\lim_{n\to \infty}\int_Q \partial_{tt}\bu^n \cdot \bv g_k(|\bT^n|)+ \bT^n \cdot \nabla (\bv g_k(|\bT^n|)) - \bef \cdot \bv g_k(|\bT^n|) \dx\dt\\
&\quad -\lim_{k\to \infty}\lim_{n\to \infty}\int_Q  \bT^n \cdot (\nabla g_k(|\bT^n|)\otimes \bv )\dx \dt\\
&=-\lim_{k\to \infty}\lim_{n\to \infty}\int_Q  \bT^n \cdot (\nabla g_k(|\bT^n|)\otimes \bv )\dx \dt,
\end{split}
\end{equation}
where for the last equality we have used \eqref{WFn} with $\bw:=\bv g_k(|\bT^n|)$. It remains to show that the right-hand side of \eqref{last2} vanishes. We define
$$
M_{k,n}(s):= \int_0^s \frac{g'_k(t)}{\frac{\phi'(t)}{t}+n^{-1}}\dt \le \int_0^s \frac{t g'_k(t)}{\phi'(t)}\dt=:M_k(s).
$$
Then, using that $|g_k'(s)|\le Cs^{-1} \chi_{\{s\in (k,2k)\}}$, we see that
\begin{equation}\label{prym}
M_k(s)\left\{\begin{aligned}&\le C\min\{s,k\} &&\textrm{for all } s\ge 0,\\
&=0 &&\textrm{for }s\le k.\end{aligned}\right.
\end{equation}
Next we can use the structural assumption \eqref{A4} to rewrite the term under the limit in \eqref{last2} as
\begin{equation}\label{smus}
\begin{aligned}
-\int_Q&  \bT^n \cdot (\nabla g_k(|\bT^n|)\otimes \bv )\dx \dt \\&= -\int_Q  \bG_n(\bT^n) \cdot (\nabla |\bT^n|\otimes \bv )\frac{g'_k(|\bT^n|)}{\frac{\phi'(|\bT^n|)}{|\bT^n|}+n^{-1}}\dx \dt\\
&= -\int_Q  \bG_n(\bT^n) \cdot (\nabla M_{k,n}(|\bT^n|)\otimes \bv )\dx \dt\\
&= \int_Q  \diver \bG_n(\bT^n) \cdot \bv  M_{k,n}(|\bT^n|) \dx \dt+\int_Q  \bG_n(\bT^n) \cdot \nabla \bv  M_{k,n}(|\bT^n|)\dx \dt.
\end{aligned}
\end{equation}
For the first term on the right-hand side of (\ref{smus}), we use the definition of $\mA_n$ alongside the Cauchy--Schwarz inequality to obtain
$$
\begin{aligned}
&| \diver \bG_n(\bT^n) \cdot \bv  M_{k,n}(|\bT^n|) |=\left|\sum_{i,j,a,b=1}^d (\mA_n(\bT^n))^{ij}_{ab} \partial_j \bT_{ab}^n \bv_i M_{k,n}(|\bT^n|) \right|\\
&=\left|\sum_{m=1}^d \sum_{i,j,a,b=1}^d (\mA_n(\bT^n))^{ij}_{ab} \partial_m \bT_{ab}^n \delta_{mj}\bv_i M_{k,n}(|\bT^n|) \right|\\
&\le \left|\sum_{m=1}^d \left(\partial_m \bT^n, \partial_m \bT^n\right)^{\frac12}_{\mA_n(\bT^n)}\left(\sum_{i,j,a,b=1}^d (\mA_n(\bT^n))^{ij}_{ab} \delta_{mj}\bv_i\delta_{ma}\bv_b M^2_{k,n}(|\bT^n|)\right)^{\frac12} \right|\\
&\le \left|\sum_{m=1}^d \left(\partial_m \bT^n, \partial_m \bT^n\right)^{\frac12}_{\mA_n(\bT^n)}\left((n^{-1}+\frac{C}{1+|\bT^n|})|\bv|^2 M^2_{k,n}(|\bT^n|)\right)^{\frac12} \right|.
\end{aligned}
$$
Using this bound in \eqref{smus} and then in \eqref{last2}, recalling the fact that $\bv$ is compactly supported, we deduce with the help of H\"{o}lder's inequality and the uniform bound \eqref{sp-r-ln} that
 \begin{equation}\label{last12}
\begin{split}
|I|&\le \lim_{k\to \infty}\lim_{n\to \infty} \int_Q \left|\sum_{m=1}^d \left(\partial_m \bT^n, \partial_m \bT^n\right)^{\frac12}_{\mA_n(\bT^n)}\left(\left( n^{-1}+\frac{C}{1+|\bT^n|}\right) |\bv|^2 M^2_{k,n}(|\bT^n|)\right)^{\frac12} \right| \dx\dt\\
&\le C(\bv)\lim_{k\to \infty}\lim_{n\to \infty} \left(\int_Q \left(n^{-1}+\frac{C}{1+|\bT^n|} \right) M^2_{k,n}(|\bT^n|) \dx\dt\right)^{\frac12}\\
&=C(\bv)\lim_{k\to \infty} \left(\int_Q \frac{ M^2_{k}(|\bT|)}{|\bT|} \dx\dt\right)^{\frac12},
\end{split}
\end{equation}
where for the last equality we used \eqref{Kn9} and the boundedness of $M_k$. Consequently,  using that $\bT\in L^1(Q;\R^{d\times d})$ and  the structure of $M_k$ \eqref{prym}, we deduce that
\begin{equation*}%\label{last12}
\begin{split}
|I|&\le C(\bv)\lim_{k\to \infty} \left(\int_Q \frac{ M^2_{k}(|\bT|)}{|\bT|} \dx\dt\right)^{\frac12}\le C(\bv)\lim_{k\to \infty} \left(\int_{Q\cap \{|\bT|>k\}} |\bT| \dx\dt\right)^{\frac12}=0.
\end{split}
\end{equation*}
Since $\bv$ is arbitrary, we see that \eqref{WF2} holds for almost all $t\in (0,T)$ and all smooth compactly supported $\bw$. Finally, using a weak$^*$ density argument based on \cite[Lemma A.3]{BeBuMaSu17} we deduce that \eqref{WF2} holds for an arbitrary $\bw\in W^{1,2}_0(\Omega,\R^d)$ fulfilling $\beps(\bw)\in L^{\infty}(Q;\R^{d\times d})$. This concludes the proof of the existence of a solution as asserted in Theorem~\ref{T4.1}.

\subsection{Uniqueness of the solution}
It remains to prove the uniqueness of such weak solutions.
Let  $(\bu_1,\bT_1)$ and $(\bu_2,\bT_2)$ be two solutions emanating from the same data and denote $\bv:=\bu_1-\bu_2$. Then it follows from \eqref{WF2} that, for almost all $t\in (0,T)$  and for every $\bw\in W^{1,\infty}_0(\Omega;\mathbb{R}^d)$,
\begin{equation}\label{uniq1}
\int_{\Omega} \partial_{tt} \bv \cdot \bw + (\bT_1-\bT_2)\cdot \beps(\bw)\dx =0.
\end{equation}
Since $\partial_t \beps(\bv)$ and $\beps(\bv)$ belong to $L^{\infty}(\Omega;\R^{d\times d})$ for almost all $t\in (0,T)$, we can again use the weak$^*$ density argument as in the previous section to deduce that \eqref{uniq1} holds with $\bw:=\alpha \bv + \beta \partial_t \bv$. Consequently, since we have
$$
\alpha \bv + \beta \partial_t \bv=\bG(\bT_1)-\bG(\bT_2),
$$
we can use the monotonicity of $\bG$ and  integration over $(t_0, t)$, with $0<t_0<t < T$, to deduce from \eqref{uniq1} that
$$
\begin{aligned}
0&\ge 2\int_{t_0}^t \int_{\Omega}\partial_{tt} \bv \cdot (\alpha \bv + \beta \partial_t \bv)\dx \dtau \\
&\quad= \beta \int_{\Omega}|\partial_t \bv(t)|^2 - |\partial_t \bv(t_0)|^2 +2\alpha \partial_t \bv(t)\cdot \bv(t) - 2\alpha \bv(t_0)\cdot \bv(t_0)\dx -2\alpha \int_{t_0}^t\int_{\Omega} |\partial_t \bv|^2\dx \dtau.
\end{aligned}
$$
We note that this procedure is rigorous for every such $t_0>0$ thanks to the regularity of $\bu_1$ and $\bu_2$ asserted in \eqref{FSu2}.
Since $\bv\in \mathcal{C}^1([0,T]; L^2(\Omega;\R^d))$ as a result of  \eqref{FSu2}, we can use \eqref{bcint2} and let $t_0\to 0_+$ in the above inequality in order to deduce that
$$
\begin{aligned}
0&\ge \beta \int_{\Omega}|\partial_t \bv(t)|^2  +2\alpha \partial_t \bv(t)\cdot \bv(t) \dx -2\alpha \int_{0}^t\int_{\Omega} |\partial_t \bv|^2\dx \dtau\\
&=\beta \int_{\Omega}|\partial_t \bv(t)|^2  +2\alpha \partial_t \bv(t)\cdot\left( \int_0^t\partial_t \bv(\tau)\dtau\right) \dx -2\alpha \int_{0}^t\int_{\Omega} |\partial_t \bv|^2\dx \dtau\\
&\ge \frac{\beta}{2} \left(\|\partial_t \bv (t)\|_2^2 - C(\alpha,\beta,T)\int_0^t \|\partial_t \bv(\tau)\|_2^2 \dtau\right)
\\&= \mathrm{e}^{-tC(\alpha,\beta,T)}\ddt \left( \mathrm{e}^{-tC(\alpha,\beta,T)}\int_0^t \|\partial_t \bv(\tau)\|_2^2 \dtau\right).
\end{aligned}
$$
Simple integration with respect to $t$ then gives that $\partial_t \bv \equiv 0$ almost everywhere in $Q$ and consequently $\bu_1=\bu_2$. By strict monotonicity, we  necessarily also have that $\bT_1=\bT_2$ almost everywhere in $Q$ and, hence, uniqueness follows.

%\newpage
%\section*{References}
\bibliographystyle{plainnat}
\bibliography{bibliography}

\end{document}